\documentclass[11pt]{article}
\usepackage[english]{babel}
\usepackage{ctex}
\usepackage{amsfonts}
\usepackage{mathrsfs}
\usepackage{bbm}
\usepackage{amsfonts}
\usepackage{multirow}
\usepackage{amssymb,amsmath,graphicx}
\usepackage{float}
\usepackage[colorlinks=true]{hyperref}
\usepackage{tikz}
\usepackage{setspace}
\usepackage{fancyhdr}
\usepackage{mathrsfs}
\usepackage{amsmath,amssymb}
\usepackage{amsthm}
\usepackage{mathtools}
\usepackage{amssymb}
\usepackage{bbm}
\usepackage{enumitem}
\usepackage{bm}
\usepackage{hyperref}
\usepackage{cite}
\usepackage{url}
\usepackage{algorithm2e}
\hypersetup{urlcolor=blue, citecolor=red}
\usepackage{pifont}
\usepackage[perpage,symbol*,bottom]{footmisc}
\DefineFNsymbols{circled}{{\ding{172}}{\ding{173}}{\ding{174}}
{\ding{175}}{\ding{176}}{\ding{177}}{\ding{178}}{\ding{179}}{\ding{180}}{\ding{181}}}
\setfnsymbol{circled}
\usepackage{scrextend}
\hypersetup{hidelinks}
\usepackage{longtable}
\usepackage{tikz}
\usepackage{array}
\usepackage{caption}
\usepackage{booktabs}
\usetikzlibrary{trees}

\allowdisplaybreaks[1]

\allowdisplaybreaks
\usepackage{cleveref}
\crefname{theorem}{Theorem}{Theorems}
\crefname{table}{Table}{Tables}
\crefname{lemma}{Lemma}{Lemmas}
\crefname{definition}{Definition}{Definitions}
\begin{document}
\title{A fractional attraction-repulsion chemotaxis system with time-space dependent growth source and nonlinear productions}
\author{%
  Liyan Song\thanks{Email: lysong0317@163.com},
  Qingchun Li\thanks{Corresponding author. Email: qingchunli\_dlut@163.com},
  Yang Cao\thanks{Email: mathcy@dlut.edu.cn}
\\[3pt]
\small \emph{School of Mathematical Sciences, Dalian University of Technology,}\\
\small \emph{Dalian 116024, China}
}
\date{}
\maketitle

{\bf Abstract}
This paper studies a fractional attraction-repulsion system with time-space dependent growth source and nonlinear productions:
\begin{equation*}
\left\{
\begin{aligned}\label{1.1}
&u_t = -(-\Delta)^\alpha u - \chi_1 \nabla \cdot (u \nabla v_1) + \chi_2 \nabla \cdot (u \nabla v_2) + a(x,t)u - b(x,t)u^\gamma, &x \in \mathbb{R}^N, \, t > 0, \\
&0 = \Delta v_1 - \lambda_1 v_1 + \mu_1 u^k, &x \in \mathbb{R}^N, \, t > 0, \\
&0 = \Delta v_2 - \lambda_2 v_2 + \mu_2 u^k, &x \in \mathbb{R}^N, \, t > 0.
\end{aligned}
\right.
\end{equation*}\\
We first establish the global boundedness of classical solutions with nonnegative bounded and uniformly continuous initial data in two different cases: $\gamma \geq k + 1$ and $\gamma < k + 1$, respectively. For a fixed $\gamma$, when $k$ exceeds the critical value $\gamma - 1$, a larger $b$ must be chosen to suppress the blow-up of the solution. Moreover, we show the persistence of the global solutions for both cases $\gamma = k + 1$ and $\gamma \neq k + 1$.


\newtheorem{theorem}{Theorem}[section]

\newtheorem{lemma}{Lemma}[section]
\newtheorem{definition}[lemma]{Definition}
\newtheorem{proposition}{Proposition}[section]
\newtheorem{corollary}{Corollary}[section]
\newtheorem{remark}{Remark}
\renewcommand{\theequation}{\thesection.\arabic{equation}}
\catcode`@=11 \@addtoreset{equation}{section} \catcode`@=12
\maketitle{}

\section{ Introduction }
The current paper is devoted to studying a fractional logistic type attraction-repulsion system with nonlinear productions
\begin{equation}
\left\{
\begin{aligned}\label{1.1}
&u_t = -(-\Delta)^\alpha u - \chi_1 \nabla \cdot (u \nabla v_1) + \chi_2 \nabla \cdot (u \nabla v_2) + a(x,t)u - b(x,t)u^\gamma, &x \in \mathbb{R}^N, \, t > 0, \\
&0 = \Delta v_1- \lambda_1 v_1 + \mu_1 u^k,  &x \in \mathbb{R}^N, \, t > 0, \\
&0 = \Delta v_2- \lambda_2 v_2 + \mu_2 u^k,  &x \in \mathbb{R}^N, \, t > 0,
\end{aligned}
\right.
\end{equation}
where $N \geq 1$, $\gamma > 1$, $\alpha \in ( \frac{1}{2}, 1 )$, $\chi_i \geq 0$, $k \geq 1$, $\lambda_i, \mu_i > 0 ~(i = 1,2)$ and \( a(x,t) \), \( b(x,t) \) are uniformly H\"{o}lder continuous in \( (x,t) \in \mathbb{R}^N \times \mathbb{R}\) with exponent \( 0 < \nu < 1 \) and satisfy the following conditions
\begin{equation}\label{0.9.0}
\begin{cases}
0 < a_{\inf} := \inf_{\substack{t \in \mathbb{R},  x \in \mathbb{R}^N}} a(x,t) \leq a_{\sup} := \sup_{\substack{t \in \mathbb{R}, x \in \mathbb{R}^N}} a(x,t) < +\infty,\\
0 < b_{\inf} := \inf_{\substack{t \in \mathbb{R},  x \in \mathbb{R}^N}} b(x,t) \leq b_{\sup} := \sup_{\substack{t \in \mathbb{R},  x \in \mathbb{R}^N}} b(x,t) < +\infty.
\end{cases}
\end{equation}

Since the pioneering studies by Keller and Segel in the 1970s \cite{KS,KSS}, extensive research has been devoted to exploring and generalizing the classical chemotaxis model \cite{NAYM,Y,YH,M,GJ,D,K}. It is now well recognized that in many biological systems, cellular behavior is influenced not only by chemoattractants but also by chemorepellent cues. Against the backdrop of this refined understanding of biological regulatory mechanisms, relevant scholars have shifted their focus to generalizing the logistic source term to better align with real-world biological complexity. Concurrently, as efforts to refine the logistic source term have intensified, researchers have been further motivated to move beyond its traditional fixed form. More flexible versions that explicitly incorporate spatiotemporal dependence have been introduced, allowing the model to capture dynamic environmental variations in biological contexts. Consider the following attraction-repulsion chemotaxis system, which explicitly includes a logistic source term varying in both space and time:
\begin{align}\label{0.2.1}
\begin{cases}
u_t = \Delta u - \chi_1 \nabla\cdot(u \nabla v_1) + \chi_2 \nabla\cdot(u \nabla v_2) + a(x,t)u - b(x,t)u^{\gamma}, & x \in \mathbb{R}^N, \, t > 0, \\
0 = \Delta v_1 - \lambda_1 v_1 + \mu_1 u, & x \in \mathbb{R}^N, \, t > 0, \\
0 = \Delta v_2 - \lambda_2 v_2 + \mu_2 u, & x \in \mathbb{R}^N, \, t > 0,
\end{cases}
\end{align}
when $a(x,t)$, $b(x,t)$ are constants, for the bounded domain, numerous studies have  investigated the dynamics of \eqref{0.2.1} with many interesting dynamical scenarios observed (see \cite{ZL,LX}). Moreover, to more accurately model the chemotactic cell movement toward and away from regions of a higher chemical signal concentration, many researchers have replaced the linear signal secretion $\mu u$ with a nonlinear production term $\mu u^k$ \cite{EM,EK,EKK,LMS,XZJ}. Nevertheless, studies concerning the whole space $\mathbb{R}^N$ within relevant domains are still comparatively limited. For example, Salako and Shen \cite{BW} investigated the asymptotic stability and spreading properties of solutions to \eqref{0.2.1} when $\gamma=2$. In contrast to the case in bounded domains, results for the Cauchy problem in the whole space involving both generalized logistic sources and nonlinear production terms are even more scarce. Notably, the work of Hassan et al. \cite{HZWY}, which established the existence of globally bounded classical solutions for a chemotaxis-consumption model with a logistic term of the form $au - bu^\gamma$, serves as a key motivation. Inspired by their findings, we introduce a generalized logistic exponent $\gamma$ into \eqref{1.1} and conduct a comparative analysis to elucidate the respective influences of $\gamma$ and the nonlinear exponent $k$ on the solution dynamics.

However, for the specific case of spatiotemporally dependent logistic sources on $\mathbb{R}^N$, a number of representative works \cite{SS1,SS2,SS3,SX1,BS1} have been conducted.
For the case without repulsion, Salako et al. \cite{SS1} focused on parabolic-elliptic chemotaxis models and analyzed the local existence, global existence, persistence and asymptotic spreading speed of the solutions. By comparing this spatiotemporally dependent logistic source model with the constant logistic source model and combining with the existing conclusions of the Fisher-KPP equation, they obtained the estimates of the asymptotic spreading speed of the solutions. In particular, when the chemotactic sensitivity coefficient \( \chi \to 0^+ \), the upper and lower bounds of the spreading speed are \( (c_-^*, c_+^*) \to (2\sqrt{a_{\inf}}, 2\sqrt{a_{\sup}}) \). This result corresponds to the spreading speed in the case of a constant logistic source, which verifies the rationality of the spatiotemporal dependence improvement. In 2020, Shen et al. \cite{SX1} investigated the global existence, persistence and spreading speed of solutions under the conditions that \( a(x,t) = r(x - ct) \), \( b(x,t) \) is a constant, and \( r(x) \) is H\"{o}lder continuous and bounded. It revealed that \( r(-\infty) < 0 < r(\infty) \) when \( b > \chi\mu \) and \( b \geq \left(1 + \frac{1}{2} \frac{(\sqrt{r^*} - \sqrt{\lambda})_+}{\sqrt{r^*} + \sqrt{\lambda}}\right)\chi\mu \). In addition, the phenomenon of species extinction occurs if the spreading speed \( c > c^* = 2\sqrt{r^*} \). When considering the chemotactic repellent signal $v_2$, Bao et al. \cite{BS1} considered the asymptotic dynamics in logistic type chemotaxis models in a free boundary or an unbounded boundary. They proved the global existence of solutions, pointwise/uniform persistence of solutions and stability of strictly positive entire solutions for the model in the unbounded domain.

Recent advances have revealed that classical chemotaxis systems fail to fully capture the complex behavioral patterns observed in various biological systems \cite{AYD,EC}. To better decribe these complex processes, researchers have recently studied the following fractional attraction-repulsion system
\begin{equation}\label{1.3}
\left\{
\begin{aligned}
&u_t = -(-\Delta)^\alpha u - \chi_1 \nabla \cdot (u \nabla v_1) + \chi_2 \nabla \cdot (u \nabla v_2) + a(x,t)u - b(x,t)u^2, & x \in \mathbb{R}^N, \, t > 0, \\
&0 = \Delta v_1- \lambda_1 v_1 + \mu_1 u, & x \in \mathbb{R}^N, \, t > 0, \\
&0 = \Delta v_2- \lambda_2 v_2 + \mu_2 u, & x \in \mathbb{R}^N, \, t > 0,
\end{aligned}
\right.
\end{equation}
when $a(x,t)$, $b(x,t)$ are constants, Zhang et al. \cite{ZZL} investigated the existence of globally bounded classical solutions and the asymptotic stability of positive constant steady states when $v_2\equiv 0$. On this basis, Jiang et al. \cite{JLL} incorporate chemotactic repulsion signals and establish the global existence, asymptotic behavior, and spreading properties of classical solutions for the system \eqref{1.3}. For the time-space dependent logistic source, Zhang et al. \cite{ZLZ1} considered a fractional parabolic-elliptic Keller-Segel system when $v_2\equiv 0$. They obtained the global boundedness and the pointwise/uniform persistence of classical solutions when $\alpha \in (\frac{1}{2},1)$.

However, there are few studies on fractional chemotaxis models with spatiotemporally dependent, generalized logistic source and nonlinear production. Inspired by \cite{BW,JLL,XZJ,LMS,SX,ZZL,SS1,BS1,ZLZ1}, we investigate the global boundedness and persistence of solutions for the fractional attraction-repulsion chemotaxis system \eqref{1.1}, which incorporates a logistic-type source and nonlinear signal production in the whole space.
The primary aim of this research is to investigate three key mechanisms in attraction-repulsion, fractional dynamics, and logistic-type sources with nonlinear production and to elucidate their role in regulating the behavior of solutions to \eqref{1.1}.
Throughout this paper, denote
\[
C_{unif}^b(\mathbb{R}^N) = \{ u \in C(\mathbb{R}^N) \mid u(x) \text{ is uniformly continuous in } x \in \mathbb{R}^N \text{and} \sup_{x \in \mathbb{R}^N} |u(x)| < \infty \}
\] provided with the norm $\| u \|_{L^\infty} = \sup_{x \in \mathbb{R}^N} | u(x) |$.

Now, we present the main results as follows.
\begin{theorem}[Global Boundedness]\label{thm:mytheorem1.2}
Assume that \eqref{0.9.0} holds. Let $\gamma > 1$, $\alpha \in ( \frac{1}{2}, 1 )$, $\chi_1, \chi_2 \geq 0$, $k \geq 1$, $\nu \in (2-2\alpha,1)$, $u_0 \in C_{unif}^b \left( \mathbb{R}^N \right)$ and $\inf\limits_{x \in \mathbb{R}^N} u_0(x) > 0$. Then \eqref{1.1} has a unique nonnegative global classical solution $(u, v_1, v_2)$ if one of the following assumptions holds:
\begin{itemize}
\item[\rm(a)] $\gamma \geq k + 1$, $b_{\inf} + \chi_2 \mu_2 - \chi_1 \mu_1 - M > 0$;
\item[\rm (b)] $\gamma < k + 1$, $\chi_2 \lambda_2 \mu_2 \geq \chi_1 \lambda_1 \mu_1$, $\lambda_1 \geq \lambda_2$;
\item[\rm(c)] $\gamma < k + 1$, $\| u_0 \|_{L^\infty} \leq ( \frac{b_{\inf} - a_{\sup}}{M + \chi_1 \mu_1} )^{\frac{1}{k}}$, $b_{\inf} > a_{\sup} + M + \chi_1 \mu_1$;
\item[\rm(d)] $\gamma \neq k + 1$, $\chi_2 \mu_2 = \chi_1 \mu_1$, $\lambda_1 = \lambda_2$,
\end{itemize}
where
\begin{align*}
M := \min \biggl\{ &\frac{1}{\lambda_1} \Bigl( (\chi_2 \mu_2 \lambda_2 - \chi_1 \mu_1 \lambda_1)_+ + \chi_2 \mu_2 (\lambda_1 - \lambda_2)_+ \Bigr), \\
&\frac{1}{\lambda_2} \Bigl( (\chi_2 \mu_2 \lambda_2 - \chi_1 \mu_1 \lambda_1)_+ + \chi_1 \mu_1 (\lambda_1 - \lambda_2)_+ \Bigr) \biggr\}.
\end{align*}
Furthermore, it holds that $\| u \|_{L^\infty} \leq C_0$, where
\begin{align*}
C_0 =
\begin{cases}
\| u_0 \|_{L^\infty}, &  if ~ \chi_i = a = b = 0, \\
\max \left\{ 1, \| u_0 \|_{L^\infty}, \left( \frac{a_{\sup}}{b_{\inf} + \chi_2 \mu_2 - \chi_1 \mu_1 - M} \right)^{\frac{1}{k}} \right\}, &  if ~ \mathrm{(a)} ~holds, \\
\max \left\{ 1, \| u_0 \|_{L^\infty}, \left( \frac{a_{\sup}}{b_{\inf}} \right)^{\frac{1}{\gamma - 1}} \right\}, & if ~ \mathrm{(b)} ~holds, \\
\max \left\{ 1, \| u_0 \|_{L^\infty}, C_* \right\}, & if ~ \mathrm{(c)} ~holds, \\
\max \left\{1, \| u_0 \|_{L^\infty}, \left( \frac{a_{\sup}}{b_{\inf}} \right)^{\frac{1}{\gamma - 1}} \right\}, & if ~ \mathrm{(d)} ~holds,
\end{cases}
\end{align*}
where \( C_* \leq \left( \frac{b_{\inf} - a_{\sup}}{M + \chi_1 \mu_1} \right)^{\frac{1}{k}} \).
If either
\begin{align*}
\gamma = k + 1, ~ b_{\inf} + \chi_2 \mu_2 - \chi_1 \mu_1 - M > 0
\end{align*}
or
\begin{align*}
\gamma \neq k + 1, ~ \chi_2 \mu_2 = \chi_1 \mu_1, \lambda_1 = \lambda_2,
\end{align*}
then

\noindent (i) for any \( u_0 \in C_{unif}^b(\mathbb{R}^N) \) and \( \inf_{x \in \mathbb{R}^N} u_0(x) > 0 \) with \( t_0 \in \mathbb{R} \), we can deduce that
\[
\begin{cases}
\liminf_{t \to \infty} \inf_{x \in \mathbb{R}^N} u(x ,t + t_0; t_0, u_0) \leq \left( \frac{a_{\sup} + M C_0^k}{b_{\inf} + \chi_2 \mu_2 - \chi_1 \mu_1} \right)^{\frac{1}{k}}, & \gamma = k + 1, \\
\liminf_{t \to \infty} \inf_{x \in \mathbb{R}^N} u(x ,t + t_0; t_0, u_0) \leq \left( \frac{a_{\sup}}{b_{\inf}} \right)^{\frac{1}{\gamma - 1}}, & \gamma \neq k + 1
\end{cases}
\]
and
\[
\begin{cases}
\left( \frac{a_{\inf}}{b_{\sup} + \chi_2 \mu_2} \right)^{\frac{1}{k}} \leq \limsup_{t \to \infty} \sup_{x \in \mathbb{R}^N} u(x ,t + t_0; t_0, u_0), & \gamma = k + 1, \\
\left( \frac{a_{\inf}}{b_{\sup}} \right)^{\frac{1}{\gamma - 1}} \leq \limsup_{t \to \infty} \sup_{x \in \mathbb{R}^N} u(x ,t + t_0; t_0, u_0), & \gamma \neq k + 1.
\end{cases}
\]
(ii) for every positive real number $M>0$, there is a constant $K$ such that for every $u_0 \in C_{{unif}}^b(\mathbb{R}^N)$ with $0\leq u_0 \leq C_0$, it follows that
\[
\| v_i(x ,t + t_0; t_0, u_0) \|_{C_{{unif}}^{1,\nu}(\mathbb{R}^N)} \leq K,~ i=1,2.
\]
(iii) let \( t_0 \in \mathbb{R} \) and \( u_0 \in C_{{unif}}^b(\mathbb{R}^N) \) be given such that \( u_0 \neq 0 \) and the solution \( u(x,t; t_0, u_0) \) is defined for all \( t \in [t_0, \infty) \) and satisfies \( \limsup\limits_{t \to \infty} \| u(\cdot ,t ; t_0, u_0) \|_{L^\infty} < \infty \), we have
\[
\begin{cases}
\limsup\limits_{t \to \infty} \| u(\cdot ,t; t_0, u_0) \|_{L^\infty} \leq \left( \frac{a_{{\sup}}}{b_{{\inf}} + \chi_2 \mu_2 - \chi_1 \mu_1 - M} \right)^{\frac{1}{k}}, & \gamma = k + 1, \\
\limsup\limits_{t \to \infty} \| u(\cdot ,t; t_0, u_0) \|_{L^\infty} \leq \left( \frac{a_{{\sup}}}{b_{{\inf}}} \right)^{\frac{1}{\gamma - 1}}, & \gamma \neq k + 1.
\end{cases}
\]
\end{theorem}
\begin{table}[H]
\centering
\small
\caption{The global boundedness of classical solutions}
\label{tab:your_label}
\begin{tabular}{
>{\centering\arraybackslash}m{4.8cm}
>{\centering\arraybackslash}m{3.8cm}
>{\centering\arraybackslash}m{5cm}
}
\toprule
& $(\mathrm{a})\ \gamma \geq k + 1$ & $(\mathrm{c})\ \gamma < k + 1$ \\
\midrule
$\begin{aligned}[c]
\chi_2\lambda_2\mu_2 &\geq \chi_1\lambda_1\mu_1, \\
\lambda_1 &\geq \lambda_2
\end{aligned}$
& $b_{\inf} > 0$
& $b_{\inf} > a_{\sup} + \chi_2\mu_2$ \\
\addlinespace[0.6em]
$\begin{aligned}[c]
\chi_2\lambda_2\mu_2 &\geq \chi_1\lambda_1\mu_1, \\
\lambda_1 &\leq \lambda_2
\end{aligned}$
&$b_{\inf} > \chi_1\mu_1 \biggl(1 - \frac{\lambda_1}{\lambda_2}\biggr)$
&$b_{\inf} > a_{\sup} + \chi_1\mu_1 \biggl(1 - \frac{\lambda_1}{\lambda_2}\biggr) + \chi_2\mu_2$ \\
\addlinespace[0.6em]
$\begin{aligned}[c]
\chi_2\lambda_2\mu_2 &\leq \chi_1\lambda_1\mu_1, \\
\lambda_1 &\geq \lambda_2
\end{aligned}$
& $b_{\inf} > \chi_1\mu_1 - \chi_2\mu_2\frac{\lambda_2}{\lambda_1}$
& $b_{\inf} > a_{\sup} + \chi_2\mu_2\biggl(1 - \frac{\lambda_2}{\lambda_1}\biggr)+ \chi_1\mu_1$ \\
\addlinespace[0.6em]
$\begin{aligned}[c]
\chi_2\lambda_2\mu_2 &\leq \chi_1\lambda_1\mu_1, \\
\lambda_1 &\leq \lambda_2
\end{aligned}$
& $b_{\inf} > \chi_1\mu_1 - \chi_2\mu_2$
& $b_{\inf} > a_{\sup} + \chi_1\mu_1$ \\
\bottomrule
\end{tabular}
\end{table}

\medskip
\noindent\textbf{Remark 1.1} As indicated in Columns 2 and 3 of \cref{tab:your_label}, for a fixed value of $\gamma$, comparisons between the scenarios where $\gamma \geq k+1$ and $\gamma < k+1$ reveal that a higher damping coefficient $b$ is necessary to prevent solution blow-up. This observation suggests that elevated values of $k$ stimulate greater secretion of chemical signaling molecules by cells, thereby enhancing cell attraction and population expansion. The inhibitory term $-bu^{\gamma}$ serves to curb cellular proliferation, necessitating a stronger damping effect via larger $b$ when $\gamma$ is insufficient to counteract the growth-driven effects associated with high $k$.

\noindent\textbf{Remark 1.2} In case (d), the global boundedness of solutions is established in the critical scenario, which in turn ensures the asymptotic stability of the constant equilibria.

\begin{theorem}[Persistence]\label{thm:mytheorem1.4}
Assume \eqref{0.9.0} hold.
Let \(\alpha \in (\frac{1}{2}, 1)\), \( k \geq 1 \) and \( \nu \in (2 - 2\alpha, 1) \).

\noindent(i) (Pointwise persistence)
If one of the following assumptions holds:
\begin{itemize}
\item[\rm(a)] $\gamma = k + 1$, $b_{\inf} + \chi_2 \mu_2 - \chi_1 \mu_1 - M > 0$;
\item[\rm(b)] $\gamma \neq k + 1$, $\chi_2 \mu_2 = \chi_1 \mu_1$, $\lambda_1 = \lambda_2$,
\end{itemize}
then for any strictly positive initial function \( u_0 \in C_{unif}^b(\mathbb{R}^N) \), there exist positive real numbers \( \bar{m}(u_0) > 0 \) and \( \bar{M}(u_0) > 0 \) such that
\[
\bar{m}(u_0) \leq u(x ,t + t_0; t_0, u_0) \leq \bar{M}(u_0)
\]
for all $t \geq 0$, $ x \in \mathbb{R}^N$, $ t_0 \in \mathbb{R}$.

\noindent(ii) (Uniform persistence) If one of the following assumptions holds:
\begin{itemize}
\item[\rm(a)] $\gamma = k + 1$, $b_{\inf} > \left(1 + \frac{a_{\sup}}{a_{\inf}}\right)\chi_1\mu_1 - \chi_2\mu_2 - M$;
\item[\rm(b)] $\gamma \neq k + 1$, $b_{\inf}^{\frac{k}{\gamma - 1}} > \frac{a_{\sup}^{\frac{k}{\gamma - 1}}}{a_{\inf}}\chi_1\mu_1$,
\end{itemize}
then for any \( t_0 \in \mathbb{R} \) and any strictly positive initial function \( u_0 \in C_{{unif}}^b(\mathbb{R}^N) \),
there exist \( 0 < \tilde{m} < \tilde{M} < \infty \) and \( T(u_0) > 0 \) such that
\[
\tilde{m} \leq u(x ,t + t_0; t_0, u_0) \leq \tilde{M}
\]
for all $t \geq T(u_0)$, $ x \in \mathbb{R}^N$, $ t_0 \in \mathbb{R}$.
\end{theorem}

\noindent\textbf{Remark 1.3} This result extends the persistence established by \cite{ZLZ1} with logistic source $\gamma = 2$, linear productions $k = 1$ and $v_2 \equiv 0$.

This paper is organized as follows. In Section 2, we prove the global existence and boundedness of classical solutions. The persistence of \eqref{1.1} is investigated in Section 3.
\section{ Global Boundedness }
This section deal with the global  boundedness of classical solutions to the system \eqref{1.1} by using
the Schauder's fixed point theorem.

\begin{proposition}[Local Existence and Uniqueness]\label{pro1.1}
Let $\gamma > 1$, $k \geq 1$, $\alpha \in ( \frac{1}{2}, 1)$, $\nu \in (2-2\alpha,1)$ and $0 \leq u_0 \in C_{unif}^b \left( \mathbb{R}^N \right)$. Then there exist $T_{\max}$ and a unique nonnegative classical solution $(u, v_1, v_2)$ on $[t_0, t_0+T_{\max})$ such that $\lim_{t \to 0^+} u(\cdot, t+t_0) = u_0$ and
\[
u \in C \left( [t_0, t_0+T_{\max}), C_{unif}^b \left( \mathbb{R}^N \right) \right) \cap C^{1,2} \left( (t_0, t_0+T_{\max}), C_{unif}^b \left( \mathbb{R}^N \right) \right).
\]
Moreover, if $T_{\max} < \infty$, then
\[
\limsup_{t \to T_{\max}} \| u(\cdot, t) \|_{L^\infty} = \infty.
\]
\end{proposition}
We give the local existence and uniqueness of classical solutions to \eqref{1.1} with nonnegative initial data \( u_0 \in C_{unif}^b \left( \mathbb{R}^N \right) \) by using the contraction mapping theorem \cite{JLL}. We omit the proof here.

First, we consider the following Cauchy problem
\begin{equation}\label{1.0.1}
\left\{
\begin{aligned}
&u_t + (-\Delta)^\alpha u = 0, \quad &&x \in \mathbb{R}^N, \, t > 0, \\
&u(x, 0) = u_0(x), \quad &&x \in \mathbb{R}^N.
\end{aligned}
\right.
\end{equation}
The solutions of \eqref{1.0.1} can be written as $u(t) = K_t^\alpha (x) * u_0(x)$. Here, \( K_t^\alpha (x) \) is a fractional heat kernel, which is denoted by
\[
K_t^\alpha (x) := t^{-\frac{N}{2\alpha}} K^\alpha \left( t^{-\frac{1}{2\alpha}} x \right),
\]
where $K^\alpha (x) := (2\pi)^{-N} \int_{\mathbb{R}^N} e^{i \xi \cdot x} e^{-|\xi|^{2\alpha}} d\xi$ \cite{ZZL,DH}. \(-(-\Delta)^\alpha + I \) is the infinitesimal generator of an analytic semigroup \( \{ T(t) \}_{t \geq 0} \) and
\begin{equation*}
\begin{aligned}
T(t)u = e^{-t} (K_t^\alpha (x) * u)(x) = \int_{\mathbb{R}^N} e^{-t} K_t^\alpha (x - y) u(y) dy
\end{aligned}
\end{equation*}
for every \( u \in X \), \( x \in \mathbb{R}^N \) and \( t \geq 0 \), where Banach space \( X = C_{unif}^b (\mathbb{R}^N) \) or \( X = L^p (\mathbb{R}^N) \). We now present several lemmas of the fractional heat kernel theory.
\begin{lemma}[\cite{ZZL}, Lemma 4.2]\label{lemma2.1.0}
Suppose that \(\{T(t)\}_{t > 0}\) is the semigroup generated by \(-(-\Delta)^{\alpha} - I\) on \(C_{unif}^{b}(\mathbb{R}^N)\). For every \(t > 0\), the operator \(T(t)\nabla \cdot\) has a unique bounded extension on \((C_{unif}^{b}(\mathbb{R}^N))^N\) satisfying
\[
\| T(t)\nabla \cdot u \|_{L^\infty} \leq C_1 t^{-\frac{1}{2\alpha}} e^{-t} \| u \|_{L^\infty}
\]
for all $u \in (C_{unif}^{b}(\mathbb{R}^N))^N$ and $t > 0$, where \(C_1\) depends only on $\alpha$ and $N$.
\label{lemma:Lemma 2.1}
\end{lemma}

\begin{lemma}\label{lemma:Lemma 2.5}
For every \( u \in C_{unif}^b(\mathbb{R}^N) \), we can deduce that
\begin{align*}
\bigl\lVert \nabla \mu_i (\Delta - \lambda_i I)^{-1} u \bigr\rVert_{L^\infty}
\leq \frac{\sqrt{N} \, \mu_i}{\sqrt{\lambda_i}} \bigl\lVert u^{k} \bigr\rVert_{L^\infty}, \, i = 1, 2.
\end{align*}
\end{lemma}
\begin{proof}
Similar to the argument in the proof of Lemma 3.3 in \cite{BW1}, we omit the details here.
\end{proof}

\begin{definition}[\cite{PR}]\label{2.3.0}
For any real-valued function \( f \) on \( \mathbb{R}^N \), define
\[
K_f^{2\alpha}(r) := \sup_{x \in \mathbb{R}^N} \int_{B_r(x)} \frac{|f(y)|}{|x - y|^{N + 1 - 2\alpha}} dy, \text{ for } r > 0,
\]
where \( B_r(x) \) denotes the open ball centered at \( x \in \mathbb{R}^N \) with radius \( r \). Then \( f \) is said to belong to the Kato class \( K^{2\alpha - 1} \) if \( \lim_{r \to 0} K_f^{2\alpha}(r) = 0 \).
\label{def: Dedinition 1}
\end{definition}

\begin{lemma}[\cite{DH}, Exercise \(4^*\), page 190]\label{lemma:Lemma 2.99}
Suppose that \(a_1\), \(a_2\), \(\alpha_1\) and \(\alpha_2\) are nonnegative constants, \(0 \leq \alpha_1, \alpha_2 < 1\) and \(0 < T < \infty\). There exists a constant \(C(\alpha_1, \alpha_2, T) < \infty\) so that for any integrable function \(u : [0, T] \to \mathbb{R}\) satisfying
\[
0 \leq u(t) \leq a_1 t^{-\alpha_1} + a_2 \int_0^t (t - \tau)^{-\alpha_2} u(\tau) d\tau
\]
for a.e. \(t \in [0, T]\), we have
\[
0 \leq u(t) \leq \frac{a_1 t^{-\alpha_1}}{1 - \alpha_1} C(\alpha_2, a_2, T), \, \text{a.e. on } 0 < t < T.
\]
\end{lemma}

\noindent \textbf{Proof of Theorem 1.1.} We introduce a linear normed space \( Q = C^b_{\textit{unif}}(\mathbb{R}^N \times [t_0,T]) \) equipped with the norm
\[
\| u \|_Q := \sum_{\theta=1}^\infty \frac{1}{2^\theta} \| u \|_{L^\infty([-r,r] \times [t_0,T])},
\]
where \( T \) is a positive real number. We define a subset \( Q_0 \) of \( Q \) as
\[
Q_0 := \left\{ u \in C^b_{{unif}}(\mathbb{R}^N \times [t_0,T]) \,\bigg|\, u(\cdot, t_0) = u_0,\ 0 \leq u(x,t) \leq C_0,\ x \in \mathbb{R}^N,\ t_0 \leq t \leq T \right\}.
\]
For any $u_0 \in Q_0$, we have $\|u\|_Q \leq C_0$. Given that the initial data $u_0 \in Q_0$ satisfies $\inf_{x \in \mathbb{R}^N} u_0(x) > 0$, the classical solution $(u, v_1, v_2)$ to \eqref{1.1} exists for $t \in [t_0, t_0 + T_{\text{max}})$. For any $x \in \mathbb{R}^N$, we have
\begin{equation*}
\begin{split}
u_t &= -(-\Delta)^\alpha u + \nabla(\chi_2 v_2 - \chi_1 v_1)\cdot \nabla u \\
&\quad + u\left( a(x,t) + \chi_2 \lambda_2 v_2 - \chi_1 \lambda_1 v_1 - b(x,t) u^{\gamma-1} + (\chi_1 \mu_1 - \chi_2 \mu_2) u^k \right),
\end{split}
\end{equation*}
where \( v_i \) satisfies $v_i = (\lambda_i I - \Delta)^{-1} \mu_i u^k$, $i = 1, 2$.

Let \( U(x,t,u):=U(x,t) \) be the solution of the following initial value problem
\begin{equation*}
\begin{cases}
\begin{aligned}
U_t &= -(-\Delta)^\alpha U + \nabla\left(\chi_2 v_2 - \chi_1 v_1\right)\cdot \nabla U \\
&\quad + U\Bigl( a(x,t) + \chi_2 \lambda_2 v_2 - \chi_1 \lambda_1 v_1 - b(x,t) U^{\gamma-1} + (\chi_1 \mu_1 - \chi_2 \mu_2) U^k \Bigr),
\end{aligned} \\[6pt]
U(\cdot, t_0) = u_0.
\end{cases}
\end{equation*}
For every \( u \in Q_0 \), we derive that
\begin{align*}
(\chi_2\lambda_2v_2 - \chi_1\lambda_1v_1)(x,t) &= \int_0^\infty \int_{\mathbb{R}^N}
\frac{e^{-\frac{|x-z|^2}{4s}}}{(4\pi s)^{\frac{N}{2}}}
\left(
\chi_2\lambda_2\mu_2 e^{-\lambda_2 s}
- \chi_1\lambda_1\mu_1 e^{-\lambda_2 s}
\right. \\
&\quad \left.
+ \chi_1\lambda_1\mu_1 e^{-\lambda_2 s}
- \chi_1\lambda_1\mu_1 e^{-\lambda_1 s}
\right)
u^k(z,t) \, dz \, ds \\
&\leq (\chi_2\lambda_2\mu_2 - \chi_1\lambda_1\mu_1)_+ C_0^k
\int_0^\infty \int_{\mathbb{R}^N}
\frac{e^{-\frac{|x-z|^2}{4s}}}{(4\pi s)^{\frac{N}{2}}}
e^{-\lambda_2 s} \, dz \, ds \\
&\quad + \chi_1\lambda_1\mu_1 C_0^k
\int_0^\infty \int_{\mathbb{R}^N}
\frac{e^{-\frac{|x-z|^2}{4s}}}{(4\pi s)^{\frac{N}{2}}}
\left( e^{-\lambda_2 s} - e^{-\lambda_1 s} \right)_+ \, dz \, ds \\
&\leq \frac{C_0^k}{\lambda_2}
\left[
(\chi_2\lambda_2\mu_2 - \chi_1\lambda_1\mu_1)_+
+ \chi_1\mu_1 (\lambda_1 - \lambda_2)_+
\right].
\end{align*}
Similarly, we have
\begin{align*}
(\chi_2\lambda_2 v_2 - \chi_1\lambda_1 v_1)(x,t) \leq \frac{C_0^k}{\lambda_1} \left[ (\chi_2\lambda_2\mu_2 - \chi_1\lambda_1\mu_1)_+ + \chi_2\mu_2 (\lambda_1 - \lambda_2)_+ \right].
\end{align*}
\noindent Hence, we have \( (\chi_2\lambda_2 v_2 - \chi_1\lambda_1v_1)(x,t) \leq M C_0^k \), where \( M \) is given in \cref{thm:mytheorem1.2}.
\noindent Then for any \( u \in Q_0 \), we obtain
\begin{equation*}
\begin{aligned}
U_t \leq -(-\Delta)^\alpha U + \nabla(\chi_2 v_2 - \chi_1v_1)\cdot \nabla U + U\left( a_{\sup} + M C_0^k - b_{\inf} U^{\gamma-1} + (\chi_1\mu_1 - \chi_2\mu_2) U^k \right).
\end{aligned}
\end{equation*}

Next, we will show that \( U(x,t) \in Q_0 \). We divide the proof into four cases.

\noindent{\textbf{Case 1}. \( \gamma \geq k+1 \)}.

Let \( C_0 := \max\left\{1, \| u_0 \|_{L^\infty}, \left( \frac{a_{\sup}}{b_{\inf} + \chi_2\mu_2 - \chi_1\mu_1 - M} \right)^{\frac{1}{k}} \right\} \). Since \( b_{\inf} + \chi_2\mu_2 - \chi_1\mu_1 - M > 0 \), we have
\[
a_{\sup} + (M + \chi_1\mu_1 - \chi_2\mu_2)C_0^k - b_{\inf} C_0^{\gamma - 1} \leq a_{\sup} + (M + \chi_1\mu_1 - \chi_2\mu_2 - b_{\inf})C_0^k \leq 0.
\]
For any $u \in Q_0$, according to the comparison principle of parabolic equations, we obtain
\[
U(x,t) \leq C_0
\]
for all $t \in [t_0,T]$, $x \in \mathbb{R}^N$.\\
\textbf{Case 2.} \( \gamma \leq k + 1 \).

Let \( C_0 := \max\left\{ 1, \| u_0 \|_{L^\infty}, \left( \frac{a_{\sup}}{b_{\inf}} \right)^{\frac{1}{\gamma - 1}} \right\} \). Due to \( \chi_2\lambda_2\mu_2 > \chi_1\lambda_1\mu_1 \) and \( \lambda_1 > \lambda_2 \), we can derive that \( M = \chi_2\mu_2 - \chi_1\mu_1 \) and
\[
a_{\sup} + (M + \chi_1\mu_1 - \chi_2\mu_2)C_0^k - b_{\inf} C_0^{\gamma - 1} = a_{\sup} - b_{\inf} C_0^{\gamma - 1} \leq 0.
\]
For any $u \in Q_0$, we apply the comparison principle of parabolic equations to obtain
\[
U(x,t) \leq C_0
\]
for all $t \in [t_0,T]$, $x \in \mathbb{R}^N$.

\noindent \(\textbf{Case 3.}\) \( \gamma \leq k + 1 \).

Let \( C_0 := \max\left\{ 1, \| u_0 \|_{L^\infty}, C_* \right\} \), \( C_* \leq \left( \frac{b_{\inf} - a_{\sup}}{M + \chi_1\mu_1} \right)^{\frac{1}{k}} \). Since \( \| u_0 \|_{L^\infty} \leq \left( \frac{b_{\inf} - a_{\sup}}{M + \chi_1\mu_1} \right)^{\frac{1}{k}} \) and \( b_{\inf} > a_{\sup} + M + \chi_1\mu_1 \), we have
\[
a_{\sup} + (M + \chi_1\mu_1)C_0^k - b_{\inf} C_0^{\gamma - 1} \leq a_{\sup} + (M + \chi_1\mu_1)C_0^k - b_{\inf} \leq 0.
\]
For any $u \in Q_0$, using the comparison principle of parabolic equations, we have
\[
U(x,t) \leq C_0
\]
for all $t \in [t_0,T]$, $x \in \mathbb{R}^N$.

\noindent \(\textbf{Case 4.}\) \( \gamma \neq k + 1 \).

Let \( C_0 := \max\left\{ 1, \| u_0 \|_{L^\infty}, \left( \frac{a_{\sup}}{b_{\inf}} \right)^{\frac{1}{\gamma - 1}} \right\} \). Due to \( \chi_2\mu_2 = \chi_1\mu_1 \), \( \lambda_1 = \lambda_2 \), we have \( M = 0 \) and
\[
a_{\sup} + (M + \chi_1\mu_1 - \chi_2\mu_2)C_0^k - b_{\inf} C_0^{\gamma - 1} = a_{\sup} - b_{\inf} C_0^{\gamma - 1} \leq 0.
\]
For any $u \in Q_0$, we conclude from the comparison principle of parabolic equations that
\[
U(x,t) \leq C_0
\]
for all $t \in [t_0,T]$, $x \in \mathbb{R}^N$.
Thus, for any \( u \in Q_0 \), we have \( U(x,t) \in Q_0 \).

It can be known from the proof process in \cite{JLL} that the mapping \( u \mapsto U(x,t,u) \in Q_0 \) is both compact and continuous. By using the Schauder fixed point theorem, we can know that there exists a fixed point \( u^* \). Thus, \( (u^*, v_1(\cdot,\cdot, u^*), v_2(\cdot,\cdot, u^*)) \) is a classical solution of \eqref{1.1}. In light of the local existence of solutions, we have \( T_{\text{max}} \geq T \) and \( u(\cdot,t, u_0) = u^* \). Since \( T > 0 \) is arbitrary, it follows that \( T_{\text{max}} = \infty \). Next, we present the proofs of the remaining conclusions of \cref{thm:mytheorem1.2}.

\noindent \textbf{Claim 1.} we prove
\begin{align}\label{0.9.1}
\limsup\limits_{t \to \infty} \| u(t, \cdot; t_0, u_0) \|_{L^\infty} \leq M^{+},
\end{align}
where
\begin{align}\label{0.9.1.0}M^{+}=
\begin{cases}
 \left( \frac{a_{\text{sup}}}{b_{{\inf}} + \chi_2 \mu_2 - \chi_1 \mu_1 - M} \right)^{\frac{1}{k}}, & \gamma = k + 1, \\
 \left( \frac{a_{\text{sup}}}{b_{{\inf}}} \right)^{\frac{1}{\gamma - 1}}, & \gamma \neq k + 1.
\end{cases}
\end{align}

We discuss it in two different cases.

\noindent Case 1. \( \gamma = k + 1 \).

Let
\[
\overline{u} = \limsup\limits_{t \to \infty} \sup\limits_{x \in \mathbb{R}^N} u(x, t; t_0, u_0).
\]
Assume that \( \overline{u} < \infty \), for any \( \varepsilon > 0 \), there exists a \( T_\varepsilon > 0 \) such that
\[
(\chi_2 \lambda_2 v_2 - \chi_1 \lambda_1 v_1)(x, t; t_0, u_0) \leq M (\overline{u} + \varepsilon)^k.
\]
Therefore, it follows that
\[
\begin{aligned}
u_t &= -(-\Delta)^\alpha u + \nabla (\chi_2 v_2 - \chi_1 v_1)\cdot \nabla u\\
&\quad + u \left( a(x, t) + (\chi_2 \lambda_2 v_2 - \chi_1 \lambda_1 v_1) - \left( b(x, t) + \chi_2 \mu_2 - \chi_1 \mu_1 \right) u^k \right) \\
&\leq -(-\Delta)^\alpha u + \nabla (\chi_2 v_2 - \chi_1 v_1)\cdot \nabla u + u \left( a_{{\sup}} +M (\overline{u} + \varepsilon)^k - \left( b_{{\inf}} + \chi_2 \mu_2 - \chi_1 \mu_1 \right) u^k \right)
\end{aligned}
\]
for any \( t \geq t_0 + T_\varepsilon \). Let \( \hat{U}(t) \) be a solution to the following ODE
\[
\begin{cases}
\hat{U}' = \hat{U} \left( a_{\text{sup}} + M (\overline{u} + \varepsilon)^k - ( b_{{\inf}} + \chi_2 \mu_2 - \chi_1 \mu_1 \right) \hat{U}^k ), \\
\hat{U}(t_0 + T_\varepsilon) = \| u(\cdot ,t_0 + T_\varepsilon; t_0, u_0) \|_{L^\infty}.
\end{cases}
\]
According to the comparison principle for parabolic equations, we can obtain that
\[
u(x, t + t_0; t_0, u_0) \leq \hat{U}(t)
\]
for all $t \geq t_0 + T_\varepsilon$, $x \in \mathbb{R}^N$. It is easy to know that
\[
\lim_{t \to \infty} \hat{U}(t) = \left( \frac{a_{{\sup}} + M (\overline{u} + \varepsilon)^k}{b_{{\inf}} + \chi_2 \mu_2 - \chi_1 \mu_1} \right)^{\frac{1}{k}},
\]
then
\[
\overline{u} \leq \left( \frac{a_{\text{sup}} + M (\overline{u} + \varepsilon)^k}{b_{\text{inf}} + \chi_2 \mu_2 - \chi_1 \mu_1} \right)^{\frac{1}{k}}.
\]
A simple calculation shows that
\[
\limsup_{t \to \infty} \| u(\cdot ,t; t_0, u_0) \|_{L^\infty} \leq \left( \frac{a_{\text{sup}}}{b_{\text{inf}} + \chi_2 \mu_2 - \chi_1 \mu_1 - M} \right)^{\frac{1}{k}}.
\]
Case 2. \( \gamma \neq k + 1 \).

Similar to the processes for Case 1, we omit the details herein.

\noindent \textbf{Claim 2.} For any \( u_0 \in C_{unif}^b(\mathbb{R}^N) \) and \( \inf_{x \in \mathbb{R}^N} u_0(x) > 0 \) with \( t_0 \in \mathbb{R} \), it follows that
\[
\begin{cases}
\left( \frac{a_{\inf}}{b_{\sup} + \chi_2 \mu_2} \right)^{\frac{1}{k}} \leq \limsup_{t \to \infty} \sup_{x \in \mathbb{R}^N} u(x, t + t_0; t_0, u_0), & \gamma = k + 1, \\
\left( \frac{a_{\inf}}{b_{\sup}} \right)^{\frac{1}{\gamma - 1}} \leq \limsup_{t \to \infty} \sup_{x \in \mathbb{R}^N} u(x, t + t_0; t_0, u_0), & \gamma \neq k + 1
\end{cases}
\]
and
\[
\begin{cases}
\liminf_{t \to \infty} \inf_{x \in \mathbb{R}^N} u(x, t + t_0; t_0, u_0) \leq \left( \frac{a_{\sup} + M C_0^k}{b_{\inf} + \chi_2 \mu_2 - \chi_1 \mu_1} \right)^{\frac{1}{k}}, & \gamma = k + 1, \\
\liminf_{t \to \infty} \inf_{x \in \mathbb{R}^N} u(x, t + t_0; t_0, u_0) \leq \left( \frac{a_{\sup}}{b_{\inf}} \right)^{\frac{1}{\gamma - 1}}, & \gamma \neq k + 1.
\end{cases}
\]

We discuss it in two different cases.

\noindent Case 1. \( \gamma = k + 1 \).

From \eqref{1.1}, it follows that
\begin{align*}
&\quad u_t + (-\Delta)^{\alpha} u - \nabla (\chi_2 v_2 - \chi_1 v_1)\cdot \nabla u \\
&= u \left( a(x, t) + (\chi_2 \lambda_2 v_2 - \chi_1 \lambda_1 v_1) - \left( b(x, t) + \chi_2 \mu_2 - \chi_1 \mu_1 \right) u^k \right) \\
&\geq u \left( a_{\inf} - \chi_1 \lambda_1 \frac{\mu_1}{\lambda_1} \| u \|_{L^{\infty}}^k - \left( b_{\sup} + \chi_2 \mu_2 - \chi_1 \mu_1 \right) \| u \|_{L^{\infty}}^k \right) \\
&\geq u \left( a_{\inf} - \left( b_{\sup} + \chi_2 \mu_2 \right) \| u \|_{L^{\infty}}^k \right).
\end{align*}
According to the comparison principle for parabolic equations, we get
\[
u(x,t + t_0; t_0, u_0) \geq e^{\int_{t_0}^{t + t_0} \left( a_{\inf} - (b_{\sup} + \chi_2 \mu_2) \| u \|_{L^{\infty}}^k \right) ds} u_{0\inf}
\]
for all $t \geq 0$. Furthermore, since \( u_{0\inf} > 0 \), we have $ \inf_{x \in \mathbb{R}^N} u(x, t + t_0; t_0, u_0) \geq 0$ for all $t \geq 0 $. For any \( \varepsilon > 0 \), there exists \( T_{\varepsilon} > 0 \) such that
\begin{align}\label{eq:2.26}
u(x, t + t_0; t_0, u_0) &\leq u^{\infty} + \varepsilon.
\end{align}
Therefore, applying the comparison principle for elliptic equations, we deduce that
\begin{align}\label{eq:2.27}
v_i(x, t + t_0; t_0, u_0) &\leq \frac{\mu_i}{\lambda_i} (u^{\infty} + \varepsilon)^k, \, i = 1, 2.
\end{align}
where \( u^{\infty} = \limsup_{t \to \infty} \sup_{x \in \mathbb{R}^N} u(x, t + t_0; t_0, u_0) \). Combining \eqref{eq:2.26} with \eqref{eq:2.27}, we then have
\[
u_t + (-\Delta)^{\alpha} u - \nabla (\chi_2 v_2 - \chi_1 v_1) \cdot\nabla u \geq u \left( a_{\inf} - (b_{\sup} + \chi_2 \mu_2) (u^{\infty} + \varepsilon)^k \right).
\]
According to the comparison principle for parabolic equations, it follows that
\begin{align}\label{0.9.3}
u(x, t + t_0; t_0, u_0) \geq e^{\left( a_{\inf} - (b_{\sup} + \chi_2 \mu_2) (u^{\infty} + \varepsilon)^k \right) (t - T_{\varepsilon})} \inf_{x \in \mathbb{R}^N} u(x,T_{\varepsilon} + t_0; t_0, u_0)
\end{align}
for all $t \geq T_{\varepsilon}$. Due to the boundedness of \( u( x,t + t_0; t_0, u_0) \), we can obtain
\[
a_{\inf} - (u^{\infty} + \varepsilon)^k (b_{\sup} + \chi_2 \mu_2) \leq 0
\]
for all $\varepsilon > 0$. The arbitrariness of \( \varepsilon \) implies that
\begin{align*}
\left( \frac{a_{\inf}}{b_{\sup} + \chi_2 \mu_2} \right)^{\frac{1}{k}} \leq \limsup_{t \to \infty} \sup_{x \in \mathbb{R}^N} u(x,t + t_0; t_0, u_0).
\end{align*}
If \( \liminf_{t \to \infty} \inf_{x \in \mathbb{R}^N} u(x,t + t_0; t_0, u_0) = 0 \), the conclusion holds obviously. Therefore, we discuss the case when \( \liminf_{t \to \infty} \inf_{x \in \mathbb{R}^N} u(x,t + t_0; t_0, u_0) > 0 \).

For any \( 0 < \varepsilon < u_{\infty} \), there exists \( T_{\varepsilon} > 0 \), we derive that
\[
u(x,t + t_0; t_0, u_0) \geq u_{\infty} - \varepsilon.
\]
Thus, by using the comparison principle for elliptic equations, we can conclude that
\[
v_i(x,t + t_0; t_0, u_0) \geq \frac{\mu_i}{\lambda_i} (u_{\infty} - \varepsilon)^k, \, i = 1, 2.
\]
Note that
\begin{align*}
&\quad u_t + (-\Delta)^{\alpha} u - \nabla (\chi_2 v_2 - \chi_1 v_1)\cdot \nabla u\\
&= u \left( a(x,t) + (\chi_2 \lambda_2 v_2 - \chi_1 \lambda_1 v_1) - \left( b(x,t) + \chi_2 \mu_2 - \chi_1 \mu_1 \right) u^k \right) \\
&\leq u \left( a_{\sup} + (\chi_2 \lambda_2 v_2 - \chi_1 \lambda_1 v_1) - \left( b_{\inf} + \chi_2 \mu_2 - \chi_1 \mu_1 \right) u^k \right) \\
&\leq u \left( a_{\sup} + M C_0^k - \left( b_{\inf} + \chi_2 \mu_2 - \chi_1 \mu_1 \right) u^k \right) \\
&\leq u \left( a_{\sup} + M C_0^k - \left( b_{\inf} + \chi_2 \mu_2 - \chi_1 \mu_1 \right) (u_{\infty} - \varepsilon)^k \right).
\end{align*}
According to the comparison principle for parabolic equations, we have
\begin{align*}
u(x,t + t_0; t_0, u_0) \leq e^{\left( a_{\sup} + M C_0^k - \left( b_{\inf} + \chi_2 \mu_2 - \chi_1 \mu_1 \right) (u_{\infty} - \varepsilon)^k \right) (t - T_{\varepsilon})} \| u( x ,T_{\varepsilon} + t_0; t_0, u_0) \|_{L^\infty}
\end{align*}
for all $t \geq T_{\varepsilon}$. Combining with \eqref{0.9.3}, we have \( a_{\sup} + M C_0^k - (u_{\infty} - \varepsilon)^k (b_{\inf} + \chi_2 \mu_2 - \chi_1 \mu_1) \geq 0 \). The arbitrariness of \( \varepsilon \) implies that
\begin{align*}
\liminf_{t \to \infty} \inf_{x \in \mathbb{R}^N} u(x,t + t_0; t_0, u_0) \leq \left( \frac{a_{\sup} + M C_0^k}{b_{\inf} + \chi_2 \mu_2 - \chi_1 \mu_1} \right)^{\frac{1}{k}}.
\end{align*}
In particular, when \( \gamma = k + 1 \), if we further specify that \( \chi_2 \mu_2 - \chi_1 \mu_1 \geq 0 \), then
\[
\liminf_{t \to \infty} \inf_{x \in \mathbb{R}^N} u(x, t + t_0; t_0, u_0) \leq \left( \frac{a_{\sup} + M C_0^k}{b_{\inf}} \right)^{\frac{1}{k}}.
\]

\noindent Case 2. \( \gamma \neq k + 1 \).

Similar to the proof for Case 1, we omit its proof here.
For the proof of $v_i(x,t + t_0; t_0, u_0) \|_{C_{{unif}}^{1,\nu}(\mathbb{R}^N)} \leq K_1$ similar to the approach in \cite{SS1}, we omit the details herein.
\section{Persistence}
In this section, we prove the pointwise persistence and uniform persistence of positive classical solutions to \eqref{1.1}. To this end, we first present several necessary lemmas, which lay the foundation for the proof of \cref{thm:mytheorem1.4}.

\begin{lemma}\label{lemma:4.2}
Let \(\alpha \in (\frac{1}{2}, 1)\) and \( \nu \in (2 - 2\alpha, 1) \).
Assume that \eqref{0.9.0} and one of the assumptions in Cases (a) and (b) of \cref{thm:mytheorem1.4} hold. For any given positive constants $M_1$ and $\varepsilon$, there exist a positive \( T > 0 \), a sufficiently large constant \( L_0 = L(M_1, T, \varepsilon) \gg 1 \) and a parameter \( \delta_0 = \delta(M_1, T, \varepsilon) \) such that for every initial function \( u_0 \in C_{{unif}}^b(\mathbb{R}^N) \) satisfying \( 0 \leq u_0 \leq M_1 \) and for all \( L \geq L_0 \),
\[
u(x,t + t_0; t_0, u_0) \leq \varepsilon,\quad \text{for} ~~0 \leq t \leq T, ~t_0 \in \mathbb{R},  ~|x|_\infty < 2L
\]
provided that \( 0 \leq u_0 \leq \delta_0 \) for \( |x|_\infty < 3L \).
\end{lemma}
\begin{proof}
According to \eqref{1.1}, we have
\begin{align}\label{6.3}
u_t &= -(-\Delta)^\alpha u + \nabla (\chi_2 v_2 - \chi_1 v_1) \cdot \nabla u + u \left( a(x, t) + \chi_2 \lambda_2 v_2 - \chi_1 \lambda_1 v_1 \right. \nonumber \\
&\quad - b(x, t) u^{\gamma - 1} + (\chi_1 \mu_1 - \chi_2 \mu_2) u^k \left. \right) \nonumber \\
&\leq -(-\Delta)^\alpha u + \nabla (\chi_2 v_2 - \chi_1 v_1)\cdot \nabla u + u \left( a_{\sup} + M C_0^k - b_{\inf} u^{\gamma - 1} + (\chi_1 \mu_1 - \chi_2 \mu_2) u^k \right).
\end{align}
This proof will be divided into two cases.

\noindent \textbf{Case 1.} \( \gamma = k + 1 \).

Equation \eqref{6.3} becomes
\[
u_t \leq -(-\Delta)^\alpha u + \nabla (\chi_2 v_2 - \chi_1 v_1) \cdot \nabla u + u \left( a_{{\sup}} + M C_0^k - (b_{{\inf}} + \chi_2 \mu_2 - \chi_1 \mu_1) u^k \right).
\]
Since \( b_{\inf} + \chi_2 \mu_2 - \chi_1 \mu_1 > M \), we can deduce that
\[
u_t \leq -(-\Delta)^\alpha u + \nabla (\chi_2 v_2 - \chi_1 v_1) \cdot \nabla u + u \left( a_{\sup} + M C_0^k \right).
\]
Let \( \overline{U}(x, t) \) be the solution of the following equation
\[
\begin{cases}
\overline{U}_t + (-\Delta)^\alpha \overline{U} + \nabla (\chi_1 v_1 - \chi_2 v_2)\cdot \nabla \overline{U} = (a_{\sup} + M C_0^k) \overline{U}, & x \in \mathbb{R}^N, t > 0, \\
\overline{U}(x, t_0) = u_0, & x \in \mathbb{R}^N.
\end{cases}
\]
By applying the comparison principle, we can know that
\begin{align}\label{6.65}
0 \leq u(x, t+t_0;t_0,u_0) \leq \overline{U}(x, t+t_0;t_0,u_0)
\end{align}
for all $x \in \mathbb{R}^N$, $t > 0$.
Let \( \overline{W} = e^{-(a_{\sup} + M C_0^k) t} \overline{U} \), we can obtain that
\[
\begin{cases}
\overline{W}_t + (-\Delta)^\alpha \overline{W} + e^{(a_{\sup} + M C_0^k) t} \nabla (\chi_1 v_1 - \chi_2 v_2) \cdot \nabla \overline{W} = 0, & t > t_0, \\
\overline{W}(x, t_0) = e^{-(a_{{\sup}} + M C_0^k) t_0} u_0.
\end{cases}
\]
According to \cref{lemma:Lemma 2.5} and \cref{2.3.0}, we can know
\[
\begin{aligned}
K_{|\nabla (\chi_1 v_1 - \chi_2 v_2)|}^{2\alpha}(r) &= \sup_{x \in \mathbb{R}^N} \int_{B_r(x)} \frac{\nabla (\chi_1 v_1 - \chi_2 v_2)}{|x - y|^{N + 1 - 2\alpha}} dy \\
&\leq \left( \frac{\chi_1 \mu_1}{\sqrt{\lambda_1}} + \frac{\chi_2 \mu_2}{\sqrt{\lambda_2}} \right) \sqrt{N} C_0^k \sup_{x \in \mathbb{R}^N} \int_{B_r(x)} \frac{1}{\tau^{N + 1 - 2\alpha}} \tau^{N - 1} d\tau \\
&\leq \left( \frac{\chi_1 \mu_1}{\sqrt{\lambda_1}} + \frac{\chi_2 \mu_2}{\sqrt{\lambda_2}} \right) \sqrt{N} C_0^k \frac{1}{2\alpha - 1} r^{2\alpha - 1}.
\end{aligned}
\]
Combining with \( \alpha \in \left( \frac{1}{2}, 1 \right) \) and \( \lim_{r \to 0} K_{|\nabla (\chi_1 v_1 - \chi_2 v_2)|}^{2\alpha}(r) = 0 \).
According to Theorem 1.1 in \cite{PR}, for any \( x, y \in \mathbb{R}^N \), \( 0 < t \leq T \), we can obtain that
\begin{align}\label{6.66}
\overline{U}(x, t) = e^{(a_{\sup} + M C_0^k)(t - t_0)} \int_{\mathbb{R}^N} P_b^\alpha(x, y, t) u_0(y) dy
\end{align}
and
\begin{align}\label{6.67}
\frac{t}{C_2 \left( t^{\frac{1}{2\alpha}} + |x - y| \right)^{N + 2\alpha}} \leq P_b^\alpha(x, y, t) \leq \frac{C_2 t}{\left( t^{\frac{1}{2\alpha}} + |x - y| \right)^{N + 2\alpha}},
\end{align}
where \( P_b^\alpha(x, t + t_0; t_0, y) \) is the fundamental solution of \( -(-\Delta)^\alpha + \chi_2 \nabla v_2 - \chi_1 \nabla v_1 \). Combining \eqref{6.66} with \eqref{6.67}, we can derive that
\begin{align*}
\overline{U}(x,t + t_0; t_0, u_0) &\leq C e^{(a_{{\sup}} + M C_0^k)(t - t_0)} \int_{\mathbb{R}^N} \frac{t}{\left[ t^{\frac{1}{2\alpha}} + |x - y| \right]^{N + 2\alpha}} u_0(y) dy \\
&\leq C e^{(a_{{\sup}} + M C_0^k)(t - t_0)} \int_{|z|_\infty \leq \frac{L}{t^{\frac{1}{2\alpha}}}} \frac{1}{(1 + |z|)^{N + 2\alpha}} u_0 \left( x + t^{\frac{1}{2\alpha}} z \right) dz \\
& \quad + C e^{(a_{{\sup}} + M C_0^k)(t - t_0)} \int_{|z|_\infty \geq \frac{L}{t^{\frac{1}{2\alpha}}}} \frac{1}{(1 + |z|)^{N + 2\alpha}} u_0 \left( x + t^{\frac{1}{2\alpha}} z \right) dz,
\end{align*}
which implies that for \( |x| < 2L \) and \( t \leq T \), we have
\begin{align*}
&\quad \overline{U}(x,t + t_0; t_0, u_0)\\
 &\leq C e^{(a_{{\sup}} + M C_0^k)(T - t_0)} \left( \delta_0 \int_{\mathbb{R}^N} \frac{1}{(1 + |z|)^{N + 2\alpha}} dz + \| u_0 \|_{L^\infty} \int_{|z|_\infty \geq \frac{L}{T^{\frac{1}{2\alpha}}}} \frac{1}{(1 + |z|)^{N + 2\alpha}} dz \right) \\
&\leq C e^{(a_{{\sup}} + M C_0^k)(T - t_0)} \left( \delta_0 \int_{0}^{\infty} \frac{1}{(1 + r)^{N + 2\alpha}} r^{N - 1} dr + \| u_0 \|_{L^\infty} \int_{|z|_\infty \geq \frac{L}{T^{\frac{1}{2\alpha}}}} \frac{1}{(1 + |z|)^{N + 2\alpha}} dz \right) \\
&\leq C e^{(a_{{\sup}} + M C_0^k)(T - t_0)} \left( \delta_0 \int_{0}^{\infty} (1 + r)^{-2\alpha - 1} dr + \| u_0 \|_{L^\infty} \int_{|z|_\infty \geq \frac{L}{T^{\frac{1}{2\alpha}}}} \frac{1}{(1 + |z|)^{N + 2\alpha}} dz \right) \\
&\leq \frac{C \delta_0}{2\alpha} e^{(a_{{\sup}} + M C_0^k)(T - t_0)} + C M_1e^{(a_{{\sup}} + M C_0^k)(T - t_0)} \int_{|z|_\infty \geq \frac{L}{T^{\frac{1}{2\alpha}}}} \frac{1}{(1 + |z|)^{N + 2\alpha}} dz.
\end{align*}
It follows from that if we set \( \delta_0 = \frac{\alpha \varepsilon}{C} e^{-(a_{{\sup}} + M C_0^k)(T - t_0)} \) and choose a sufficiently large \( L_0 \) to ensure \( \int_{|z|_\infty \geq \frac{L}{T^{\frac{1}{2\alpha}}}} \frac{1}{(1 + |z|)^{N + 2\alpha}} dz \leq \frac{\varepsilon}{2 C M_1} e^{-(a_{{\sup}} + M C_0^k)(T - t_0)} \), then we can obtain
\[
\overline{U}(x, t + t_0; t_0, u_0) \leq \varepsilon, \quad\text{for} ~ ~|x|_\infty < 2L
\]
provided that \( 0 \leq u_0 \leq \delta_0 \) for \( |x|_\infty < 3L \). Then by using \eqref{6.65}, we complete the proof.

\noindent \textbf{Case 2.} \( \gamma \neq k + 1 \).

Similar to the proof of Case 1, we omit the details herein.\\
\end{proof}

Next, let ${{a}_{0}}=\frac{{{a}_{\inf }}}{3}$, $L>0$, $D_L = \left\{ x \in \mathbb{R}^N \mid |x_i| < L,\ i = 1, 2, \ldots, N \right\}$. Consider the following problem
\begin{equation}\label{4.2}
\left\{
\begin{aligned}
  &u_{t} = -(-\Delta)^{\alpha}u + a_0 u, &\quad &x \in D_L, \\
  &u = 0,                               &\quad &x \in \partial D_L,
\end{aligned}
\right.
\end{equation}
and the corresponding eigenvalue problem
\begin{equation*}
\left\{
\begin{aligned}
  & {{(-\Delta )}^{\alpha }}u-{{a}_{0}}u=\sigma u,&\quad &x \in D_L, \\
 & u=0,                                 &\quad &x \in \partial D_L.
\end{aligned}
\right.
\end{equation*}
The eigenvalue ${{\sigma }_{L}}$ and eigenfunction ${{\phi }_{L}}$ corresponding to the operator ${{(-\Delta )}^{\alpha }}-{{a}_{0}}I$, according to the Krein-Rutman theorem, we have
\begin{equation*}
\left\{
\begin{aligned}
  &-(-\Delta)^{\alpha} \phi_L + a_0 \phi_L = \sigma_L \phi_L, \quad &&x \in D_L, \\
  &\phi_L = 0, \quad &&x \in \partial D_L,
\end{aligned}
\right.
\end{equation*}
where $0 < \phi_L(x) \le \|\phi_L(x)\|_{L^\infty} = 1$.

It is easy to know that $u(x,t)={{e}^{{{\sigma }_{L}}t}}{{\phi }_{L}}(x)$ is a solution of \eqref{4.2}. Let $u(x, t;{{u}_{0}})={{e}^{{{\sigma }_{L}}t}}{{\phi }_{L}}(x)$ be the solution of \eqref{4.2} with the initial data ${{u}_{0}} \in C({{\bar{D}}_{L}})$. Then for any $s\in \mathbb{R}$, we have
\begin{equation}\label{4.3}
\begin{aligned}
u(x, t;s{{\phi }_{L}})=s{{e}^{{{\sigma }_{L}}t}}{{\phi }_{L}}(x).
\end{aligned}
\end{equation}
By performing a transformation on the eigenvalue ${{\sigma }_{L}}$, we conclude that for any function ${{\phi }_{L}}\in {{C}^{1}}({{D}_{L}})\cap {{C}^{0}}({{\bar{D}}_{L}})$ with $\phi_L \not\equiv 0$ that vanishes outside ${{D}_{L}}$ and on its boundary, it readily yields that
\begin{equation}\label{4.7}
\begin{aligned}
\sigma_{L} := \sup_{\phi_{L}} \left[ a_{0} - \frac{
    \frac{1}{2} \int_{\mathbb{R}^{N}} \int_{\mathbb{R}^{N}}
    \frac{(\phi_{L}(x) - \phi_{L}(y))^{2}}{|x - y|^{N + 2\alpha}} \, dy \, dx
}{\int_{D_{L}} \phi_{L}^{2}(x) \, dx}
\right].
\end{aligned}
\end{equation}
Taking ${{\phi }_{L}}=1$ as a test function in ${{D}_{L}}$, we have $\underset{L\to \infty }{\mathop{\lim }}\,{{\sigma }_{L}}={{a}_{0}}>0$.
Then there exists ${{L}_{0}}\gg 0$ such that for any $L\ge {{L}_{0}}$, it holds that ${{\sigma }_{L}}>0$.

\begin{lemma}\label{3.2.0}
Let \(\alpha \in (\frac{1}{2}, 1)\) and \( \nu \in (2 - 2\alpha, 1) \).
Assume that \eqref{0.9.0} and one of the assumptions in Cases (a) and (b) of \cref{thm:mytheorem1.4} hold. For a fixed $T>0$, there exists $\delta _{0}^{*}(T)$ satisfying $0<\delta _{0}^{*}(T)<{{M}^{+}}$ such that for any $0<\delta \le \delta _{0}^{*}(T)$, the initial data ${{u}_{0}}$ satisfying $\delta \le {{u}_{0}}\le {{M}^{+}}$, then we have
\[
\delta \le u(x,t_0 + T; t_0, 0, u_0) \le M^+
\]
for all $x \in \mathbb{R}^N$ and $t_0 \in \mathbb{R}$, where $M^+$ is defined by \eqref{0.9.1.0}.
\end{lemma}
\begin{proof}
Choose $T>0$ and consider the following model
\begin{equation}\label{4.1}
\begin{cases}
u_t + (-\Delta)^\alpha u - b_\varepsilon(x, t) \cdot \nabla u = a_0 u, & \quad t_0 \le t \le t_0 + T,\ x \in D_L, \\[6pt]
u = 0, & \quad t_0 \le t \le t_0 + T,\ x \in \partial D_L
\end{cases}
\end{equation}
for any ${{t}_{0}}\le t\le {{t}_{0}}+T$, $x\in {{\bar{D}}_{L}}$, $|{{b}_{\varepsilon }}(x, t)|<\varepsilon $ and ${{\wedge }^{1-\alpha }}{{b}_{\varepsilon }}(x, t)\in {{L}^{\frac{N}{1-\alpha }}}({{D}_{L}})$, where $\wedge\triangleq(-\Delta)^{\frac{1}{2}}$. According to Lemma 2.5 in \cite{ZLZ1}, we assume that ${{u}_{{{b}_{\varepsilon, L}}}}(x,t;{{t}_{0}},{{u}_{0}})$ is the solution of \eqref{4.1} and ${{u}_{{{b}_{\varepsilon, L}}}}(x,{{t}_{0}};{{t}_{0}},{{u}_{0}})={{u}_{0}}(x)$.\\
\textbf{Step 1.} We prove that if for any $x \in {{{D}}_{L}}$, $|{{b}_{\varepsilon }}(x,t)|<\varepsilon $, there exists ${{\varepsilon }_{0}}(T)>0$ such that for any $L\ge {{L}_{0}}$, $s>0$ and $0\le \varepsilon \le {{\varepsilon }_{0}}(T)$, we have
\begin{equation}\label{4.4}
\begin{aligned}
u_{b_\varepsilon, L}\bigl(0,t_0 + T; t_0, s\phi_L\bigr)
\ge e^{\frac{T\sigma_{L_0}}{2}} s\phi_{L_0}(0)
> s\phi_{L_0}(0)
\end{aligned}
\end{equation}
and
\begin{equation}\label{4.5}
\begin{aligned}
0\le {{u}_{{{b}_{\varepsilon }},L}}(x,t+{{t}_{0}};{{t}_{0}},s{{\phi }_{L}})\le s{{e}^{{{a}_{0}}t}}
\end{aligned}
\end{equation}
for all  $0\le t\le T$, $x\in {{D}_{L}}$.
By \eqref{4.3} we know that for any $x\in {{D}_{{{L}_{0}}}}$ and $|{{b}_{\varepsilon }}(x, t)|<\varepsilon $, there exists ${{\varepsilon }_{0}}(T)>0$ such that for any $0\le \varepsilon \le {{\varepsilon }_{0}}(T)$, we have
\begin{equation}\label{4.6}
\begin{aligned}
{{u}_{{{b}_{\varepsilon }},{{L}_{0}}}}(0,{{t}_{0}}+T;{{t}_{0}}, s{{\phi }_{L}})\ge {{e}^{\frac{T{{\sigma }_{{{L}_{0}}}}}{2}}}s{{\phi }_{{{L}_{0}}}}(0).
\end{aligned}
\end{equation}
As derived in \eqref{4.7}, we can deduce that as $L$ increases, the function ${{\sigma }_{L}}$ is increasing. Therefore, for any $L\ge {{L}_{0}}$, we have ${{\sigma }_{L}}\ge {{\sigma }_{{{L}_{0}}}}$. Then we have
\[
-(-\Delta)^\alpha \phi_{L_0} + a_0 \phi_{L_0}
= \sigma_{L_0} \phi_{L_0}
\le \sigma_L \phi_{L_0}
\]
for all $x \in D_{L_0}$.

For ${{\phi }_{L}}$, we have
\[
-(-\Delta)^\alpha \phi_L + a_0 \phi_L
= \sigma_L \phi_L, ~~x \in D_{L_0}.
\]
According to the comparison principle, we can obtain
${{\phi }_{L}}(x)\ge {{\phi }_{{{L}_{0}}}}(x)$, for all $x\in {{D}_{{{L}_{0}}}}$. Note that for $t_0 \le t \le t_0 + T$, we have
\begin{equation*}
\begin{cases}
\displaystyle
\partial_t u_{b_\varepsilon, L}(\cdot,t + t_0; t_0, s\phi_L)
+(-\Delta)^\alpha u_{b_\varepsilon, L}(\cdot,t + t_0; t_0, s\phi_L) \\
\quad - b_\varepsilon(x,t) \cdot \nabla u_{b_\varepsilon, L}(\cdot,t + t_0; t_0, s\phi_L)
= a_0 u_{b_\varepsilon, L}(\cdot,t + t_0; t_0, s\phi_L),
& x \in D_{L_0}, \\[6pt]
u_{b_\varepsilon, L}(\cdot,t + t_0; t_0, s\phi_L) > 0,
& x \in \partial D_{L_0}.
\end{cases}
\end{equation*}
According to the comparison principle, we get
\[
{{u}_{{{b}_{\varepsilon, L}}}}(x, t+{{t}_{0}};{{t}_{0}},s{{\phi }_{L}})\ge {{u}_{{{b}_{\varepsilon }},{{L}_{0}}}}(x, t+{{t}_{0}};{{t}_{0}},s{{\phi }_{{{L}_{0}}}})
\]
for any $x\in {{D}_{{{L}_{0}}}}$. Combining with \eqref{4.6} can lead to \eqref{4.4}. Moreover, \eqref{4.5} can be directly obtained through the comparison principle.

\noindent \textbf{Step 2.} We discuss it in two different cases.

\noindent \textbf{Case 1.} $\gamma = k + 1$.
Consider the following model
\begin{equation}\label{4.8}
\begin{cases}
  \displaystyle
  u_t + (-\Delta)^\alpha u - b_\varepsilon(x,t) \cdot \nabla u
  = u\left(2a_0 - b(x,t)u^{\gamma - 1} - \chi_2 \mu_2 u^k\right),
  & \ x \in D_L, \\[6pt]
  u = 0,
  & \ x \in \partial D_L,
\end{cases}
\end{equation}
where ${{t}_{0}}\le t\le {{t}_{0}}+T$. We assume that ${{u}_{{{b}_{\varepsilon }},c}}(x, t;{{t}_{0}},{{u}_{0}})$ is the solution of \eqref{4.8} with ${{u}_{{{b}_{\varepsilon }},c}}(x, {{t}_{0}};{{t}_{0}},{{u}_{0}})={{u}_{0}}(x)$.

Then we prove that if for any $x\in {{\bar{D}}_{L}}$, $|{{b}_{\varepsilon }}(x,t)|<\varepsilon $, there exists ${{\varepsilon }_{0}}(T)>0$ such that for any $L\ge {{L}_{0}}$, $0 < s \leq s_0(T)$ and $0\le \varepsilon \le {{\varepsilon }_{0}}(T)$, we have
\begin{align}\label{4.9}
u_{b_{\varepsilon},c}(0,t_0 + T; t_0, s\phi_L) \ge e^{\frac{T\sigma_{L_0}}{2}} s\phi_{L_0}(0) =e^{\frac{T\sigma_{L_0}}{2}} \phi_{L_0}(0) \inf_{x \in D_L} u_0(x).
\end{align}
From \eqref{4.8}, for $t_0 \leq t \leq t_0 + T, x \in D_L$, we have
\begin{align*}
&\quad \partial_t u_{b_{\varepsilon}, L} \big(x, t; t_0, s\phi_L\big) + (-\Delta)^{\alpha} u_{b_{\varepsilon}, L} \big(x, t; t_0, s\phi_L\big) - b_{\varepsilon}(x,t) \cdot \nabla u_{b_{\varepsilon}, L} \big(x, t; t_0, s\phi_L\big) \\
&\quad - u_{b_{\varepsilon}, L} \big(x, t; t_0, s\phi_L\big) \bigg(2a_0 - b(x, t)u_{b_{\varepsilon}, L}^{\gamma - 1} \big(x, t; t_0, s\phi_L\big) - \chi_2 \mu_2 u_{b_{\varepsilon}, L}^{k} \big(x, t; t_0, s\phi_L\big) \bigg) \\
&= - u_{b_{\varepsilon}, L} \big(x, t; t_0, s\phi_L\big) \bigg(a_0 - b(x, t)u_{b_{\varepsilon}, L}^{\gamma - 1} \big(x, t; t_0, s\phi_L\big) - \chi_2 \mu_2 u_{b_{\varepsilon}, L}^{k} \big(x, t; t_0, s\phi_L\big) \bigg).
\end{align*}
Since $0 < s \leq s_0(T) := \bigg( \frac{a_0 e^{-a_0 k T}}{b_{\text{sup}} + \chi_2 \mu_2} \bigg)^{\frac{1}{k}}$, combining with \eqref{4.5}, we can get
\begin{align*}
&\quad a_0 - b(x, t)u_{b_{\varepsilon}, L}^{\gamma - 1} \big(x, t; t_0, s\phi_L\big) - \chi_2 \mu_2 u_{b_{\varepsilon}, L}^{k} \big(x, t; t_0, s\phi_L\big) \\
&\geq a_0 - b_{{\sup}} s^{k} e^{a_0 (\gamma - 1) T} - \chi_2 \mu_2 s^{k} e^{a_0 k T} \\
&\geq a_0 - (b_{{\sup}} + \chi_2 \mu_2) e^{a_0 k T} s^{k} \\
&\geq a_0 - (b_{{\sup}} + \chi_2 \mu_2) e^{a_0 k T} a_0 \frac{e^{-a_0 k T}}{b_{{\sup}} + \chi_2 \mu_2} \\
&= 0.
\end{align*}
Therefore, we can deduce that for \( t_0 \leq t \leq t_0 + T \) and \( x \in D_L \),
\begin{align*}
&\quad \partial_t u_{b_{\varepsilon}, L} \big(x, t; t_0, s\phi_L\big) + (-\Delta)^{\alpha} u_{b_{\varepsilon}, L} \big(x, t; t_0, s\phi_L\big) - b_{\varepsilon}(x, t) \cdot \nabla u_{b_{\varepsilon}, L} \big(x, t; t_0, s\phi_L\big) \\
&\leq u_{b_{\varepsilon}, L} \big(x, t; t_0, s\phi_L\big) \bigg(2a_0 - b(x, t)u_{b_{\varepsilon}, L}^{\gamma - 1} \big(x, t; t_0, s\phi_L\big) - \chi_2 \mu_2 u_{b_{\varepsilon}, L}^{k} \big(x, t; t_0, s\phi_L\big) \bigg).
\end{align*}
Let $s = \inf_{x \in D_L} u_0(x)$, then $s \leq \delta_0(T) \leq s_0(T)$. According to the comparison principle, for \( t_0 \leq t \leq t_0 + T \) and \( x \in D_L \),
\[
u_{b_{\varepsilon}, c} \big(x, t; t_0, s\phi_L\big) \geq u_{b_{\varepsilon}, L} \big(x, t; t_0, s\phi_L\big).
\]
Combining with \eqref{4.4} can lead to \eqref{4.9}.\\
\textbf{Case 2.} $\gamma \neq k + 1$.
Consider the following model
\begin{align}\label{4.8.0.0}
\begin{cases}
  \displaystyle
  u_t + (-\Delta)^\alpha u - b_\varepsilon(x,t) \cdot \nabla u
  = u\left(2a_0 - b(x,t)u^{\gamma - 1} \right),
  & \ x \in D_L, \\[6pt]
  u = 0,
  & \ x \in \partial D_L,
\end{cases}
\end{align}
where ${{t}_{0}}\le t\le {{t}_{0}}+T$. We assume that ${{u}_{{{b}_{\varepsilon }},c}}(x, t;{{t}_{0}},{{u}_{0}})$ is the solution of \eqref{4.8.0.0} with ${{u}_{{{b}_{\varepsilon }},c}}(x, {{t}_{0}};{{t}_{0}},{{u}_{0}})={{u}_{0}}(x)$.

Similar to the proof of Case 1, we can deduce that
\begin{align*}
u_{b_{\varepsilon}, \varepsilon} \big(x, t; t_0, s\phi_L\big) \geq u_{b_{\varepsilon}, L} \big(x, t; t_0, s\phi_L\big).
\end{align*}
Combining with \eqref{4.4} can lead to \eqref{4.9}.

\noindent \textbf{Step 3.} Given ${{x}_{0}}\in \mathbb{R}^N$, for any $x\in \mathbb{R}^N$, consider the following model
\begin{equation}\label{4.10}
\begin{aligned}
&\quad u_t + (-\Delta)^{\alpha} u + \nabla(\chi_1 v_1 - \chi_2 v_2) \cdot \nabla u\\
&= u\Big(a(x + x_0,t) + \chi_2 \lambda_2 v_2 - \chi_1 \lambda_1 v_1 - b(x + x_0,t) u^{\gamma - 1} + (\chi_1 \mu_1 - \chi_2 \mu_2) u^k\Big),
\end{aligned}
\end{equation}
where ${{v}_{i}}(x, t;{{t}_{0}},{{x}_{0}},{{u}_{0}})$ satisfies
$0=\Delta {{v}_{i}}-{{\lambda }_{i}}{{v}_{i}}+{{\mu }_{i}}u^{k}, i=1,2$. Let ${{\varepsilon }_{0}}(T)>0$, ${{s}_{0}}(T)>0$ following Step 1 and Step 2. Suppose that $u(x, t;{{t}_{0}},{{x}_{0}},{{u}_{0}})$ is the solution of \eqref{4.10} and $u(x,{{t}_{0}};{{x}_{0}},{{t}_{0}},{{u}_{0}})={{u}_{0}}(x)$.

Firstly, if $L\gg 1$, there exists $0<{{\delta }_{0}}(T)\le {{s}_{0}}(T)$ such that for any ${{u}_{0}}\in C_{unif}^{b}(\mathbb{R}^N)$ with $0\le {{u}_{0}}\le {{M}^{+}}$, for $|{{x}_{i}}|\le 3L,i=1,2,...,N$, ${{u}_{0}}(x)<{{\delta }_{0}}(T)$, we obtain that
\begin{equation}\label{4.11}
0 \le \lambda_1 v_1 \le \frac{a_0}{4\chi_1},~~
0 \le \lambda_2 v_2 \le \frac{a_0}{4\chi_2}, ~~
|\nabla v_1| < \frac{\varepsilon_0}{4\chi_1},~~
|\nabla v_2| < \frac{\varepsilon_0}{4\chi_2}
\end{equation}
for ${{t}_{0}}\le t\le {{t}_{0}}+T$, ${{x}_{0}}\in \mathbb{R}^N$.
Choosing $0<\varepsilon \le {{\varepsilon }_{0}}(T)$. According to \cref{lemma:4.2}, when $x\in {{D}_{3L}}$ and $0\le {{u}_{0}}(x)\le {{\delta }_{1}}$, there exist $T>0$, ${{\delta }_{1}}={{\delta }_{1}}({{M}^{+}},\varepsilon )$, ${{L}_{1}}={{L}_{1}}({{M}^{+}},T,\varepsilon )>{{L}_{0}}$ such that for any $L\ge {{L}_{1}}$, we get
\begin{equation}\label{4.12}
\begin{aligned}
u(x,t + t_0; t_0, x_0, u_0) \le \varepsilon
\end{aligned}
\end{equation}
for all $0 \le t \le T$, $t_0 \in \mathbb{R}$, $x_0 \in \mathbb{R}^N$ and $|x| \in D_{2L}$. Combining \eqref{4.11} and \eqref{4.12}, we obtain that
\begin{equation*}
\begin{aligned}
&\quad u_t + (-\Delta)^{\alpha} u + \nabla(\chi_1 v_1 - \chi_2 v_2) \cdot \nabla u\\
&= u\Big(a(x + x_0,t) -\frac{a_0}{2} - b(x + x_0,t) u^{\gamma - 1} + (\chi_1 \mu_1 - \chi_2 \mu_2) u^k\Big).
\end{aligned}
\end{equation*}
Since ${{a}_{0}}=\frac{{{a}_{\inf }}}{3}$, we get
\begin{align*}
  &\quad {{u}_{t}}+{{(-\Delta )}^{\alpha }}u+\nabla ({{\chi }_{1}}{{v}_{1}}-{{\chi }_{2}}{{v}_{2}})\cdot \nabla u\ge u\Big({{a}_{\inf }} -\frac{{{a}_{\inf }}}{6} - b(x + x_0,t) u^{\gamma - 1} + (\chi_1 \mu_1 - \chi_2 \mu_2) u^k\Big)\\
 & =u\Big(\frac{5{{a}_{\inf }}}{6}- b(x + x_0,t) u^{\gamma - 1} + (\chi_1 \mu_1 - \chi_2 \mu_2) u^k\Big) \\
 & \ge u\Big(2{{a}_{0}}-b(x + x_0,t) u^{\gamma - 1} + (\chi_1 \mu_1 - \chi_2 \mu_2) u^k\Big).
\end{align*}
Next, from \cite{ZLZ1}, we can know that ${{\wedge }^{1-\alpha }}\nabla {{v}_{i}}(x,t;{{x}_{0}},{{t}_{0}},{{u}_{0}})\in {{L}^{\frac{N}{1-\alpha }}}(\mathbb{R}^N)$ and ${{u}_{{{b}_{\varepsilon }},c}}(x,t;{{x}_{0}},{{t}_{0}},s{{\phi }_{L}})$ is the solution of \eqref{4.8}, where ${{b}_{\varepsilon }}(x,t)=\nabla ({{\chi }_{2}}{{v}_{2}}-{{\chi }_{1}}{{v}_{1}})$.
Note that for $x \in D_L$, we have
\begin{equation*}
\begin{cases}
\begin{aligned}
&u_t + (-\Delta)^{\alpha}u + \nabla(\chi_1 v_1 - \chi_2 v_2) \cdot \nabla u \ge u\Big(2a_0 - b(x + x_0,t) u^{\gamma - 1} + (\chi_1 \mu_1 - \chi_2 \mu_2) u^k\Big),
 \\[6pt]
&u(x,t;x_0, t_0, u_0) > 0,
~~~ x \in \partial D_L.
\end{aligned}
\end{cases}
\end{equation*}
Let $s=\underset{x\in {{D}_{L}}}{\mathop{\inf }}\,{{u}_{0}}(x)$, then $s\le {{\delta }_{0}}(T)\le {{s}_{0}}(T)$. According to Maximum principle, for ${{t}_{0}}\le t\le {{t}_{0}}+T$, $x\in {{D}_{L}}$, we have
\[
u(x,t;{{x}_{0}},{{t}_{0}},{{u}_{0}})\ge {{u}_{{{b}_{\varepsilon }},c}}(x,t;{{x}_{0}},{{t}_{0}},s{{\phi }_{L}}).
\]
Combining with \eqref{4.9}, it follows that
\begin{equation}\label{4.16}
\begin{aligned}
u(0,t_0 + T;x_0, t_0, u_0) \ge e^{\frac{T \sigma_{L_0}}{2}} \phi_{L_0}(0) \inf_{x \in D_L} u_0(x).
\end{aligned}
\end{equation}
\textbf{Step 4.} We prove that there exists $0<\delta _{0}^{*}(T)<\min \{{{\delta }_{0}}(T),{{M}^{+}}\}$ such that for any $0<\delta \le \delta _{0}^{*}(T)$ and any initial data ${{u}_{0}}$ satisfying $\delta \le {{u}_{0}}\le {{M}^{+}}$, we have
\[
\delta \le u(x,{{t}_{0}}+T;0,{{t}_{0}},{{u}_{0}})\le {{M}^{+}}
\]
for all $x\in \mathbb{R}^N$.

Assume that the above conclusion does not hold, that is, for ${{\delta }_{n}}\to 0$, ${{t}_{0n}}\in \mathbb{R}$, ${{\delta }_{n}}\le {{u}_{0n}}\le {{M}^{+}}$, ${{x}_{n}}\in \mathbb{R}^N$, it makes
\begin{equation}\label{4.13}
\begin{aligned}
u(x+x_{n},t;0,t_{0n},u_{0n})<\delta_{n}.
\end{aligned}
\end{equation}
Note that
\[
u(x+{{x}_{n}},t;0,{{t}_{0n}},{{u}_{0n}})=u(x,t;{{x}_{n}},{{t}_{0n}},{{u}_{0n}}(\cdot +{{x}_{n}})).
\]
Let ${{\varepsilon }_{0}}:=\varepsilon (T)>0$, ${{\delta }_{0}}:={{\delta }_{0}}(T)>0$, ${{s}_{0}}:={{s}_{0}}(T)>0$ be fixed. Let ${{D}_{0n}}=\{x\in {{R}^{N}}||{{x}_{i}}|<3L,{{u}_{0n}}(x+{{x}_{n}})>\frac{{{\delta }_{0}}}{2}\}$.
Without loss of generality, we assume that $\underset{t\to \infty }{\mathop{\lim }}\,|{{D}_{0n}}|$ exists.\\
Case 1. $\underset{t\to \infty }{\mathop{\lim }}\,|{{D}_{0n}}|=0$.\\
In this case, let ${{\{{{\tilde{u}}_{0n}}\}}_{n\ge 1}}$ be a sequence of initial values in $C_{unif}^{b}(\mathbb{R}^N)$ satisfying
\begin{equation*}
\begin{cases}
\delta_n \le \tilde{u}_{0n} \le \frac{\delta_0}{2}, \\[6pt]
\|\tilde{u}_{0n}(\cdot) - u_{0n}(\cdot + x_n)\|_{L^p(\mathbb{R}^N)} \to 0, \quad \text{as} ~~~n \to \infty
\end{cases}
\end{equation*}
for any $x \in D_{3L}$ and $p>1$.

Let
${{U}_{n}}(x,t):=u(x,t+{{t}_{0n}};{{x}_{n}},{{t}_{0n}},{{u}_{0n}}(\cdot +{{x}_{n}}))-u(x,t+{{t}_{0n}};{{x}_{n}},{{t}_{0n}},{{\tilde{u}}_{0n}})$ and ${{V}_{i,n}}(x,t):={{v}_{i}}(x,t+{{t}_{0n}};{{x}_{n}},{{t}_{0n}},{{u}_{0n}}(\cdot +{{x}_{n}}))-{{v}_{i}}(x,t+{{t}_{0n}};{{x}_{n}},{{t}_{0n}},{{\tilde{u}}_{0n}})$, ~$i=1,2$.
Therefore, ${{\{({{U}_{n}},{{V}_{i,n}})\}}_{n\ge 1}}$ satisfies
\begin{equation}\label{4.15}
\begin{cases}
\begin{aligned}
&\quad\partial_t U_n + (-\Delta)^\alpha U_n - b_n(x, t) \cdot \nabla U_n \nonumber\\
&= A_n(x, t) U_n + B_{1,n}(x, t) V_{1,n} + B_{2,n}(t,x) V_{2,n} \nonumber\\
&\quad + C_{1,n}(t,x) V_{1,n} + C_{2,n}(x, t) V_{2,n},
  && t > t_0,\ x \in \mathbb{R}^N, \nonumber\\[4pt]
&0 = \Delta V_{i,n} - \lambda_{i,n} V_{i,n} + \mu_i U_n^k,
  && t > t_0,\ x \in \mathbb{R}^N, \nonumber\\[4pt]
&U_n(x, t_0) = u_{0n}(x + x_n) - \tilde{u}_{0n}(x),
  && x \in \mathbb{R}^N,
\end{aligned}
\end{cases}
\end{equation}
where \( b_n(x, t) := \nabla(\chi_2 v_2(x + x_n ,t + t_{0n}; x_n,t_{0n}, u_{0n}(\cdot + x_n))) - \chi_1 v_1(x + x_n,t + t_{0n}; x_n,t_{0n}, u_{0n}(\cdot + x_n)) \) and
\begin{align*}
A_n(x - x_n, t - t_0) :&= a(x, t) + (\chi_1 \mu_1 - \chi_2 \mu_2)k\hat{u}^{k - 1} - b(x, t)\gamma\hat{u}^{\gamma - 1} \\
&\quad+ \chi_2 \lambda_2 v_2(x, t;x_n, t_{0n}, u_{0n}(\cdot + x_n)) - \chi_1 \lambda_1 v_1(x, t;x_n, t_{0n}, u_{0n}(\cdot + x_n)),\\
B_{1,n}(x, t) &:= -\chi_1 \lambda_1 u(x + x_n,t + t_{0n}; x_n,t_{0n}, \tilde{u}_{0n}),\\
B_{2,n}(x, t) &:= -\chi_2 \lambda_2 u(x + x_n,t + t_{0n}; x_n,t_{0n}, \tilde{u}_{0n}), \\
C_{1,n}(x, t) &:= -\chi_1 \nabla u(x + x_n,t + t_{0n}; x_n,t_{0n}, \tilde{u}_{0n}),\\
C_{2,n}(x, t) &:= \chi_2 \nabla u(x + x_n,t + t_{0n}; x_n,t_{0n}, \tilde{u}_{0n}).
\end{align*}
For ${{u}_{0n}}(\cdot, {{t}_{0}})\in {{L}^{p}}(\mathbb{R}^N)$, \eqref{4.15} has a unique solution $U(x, t;{{U}_{0n}})$ and $U(x,{{t}_{0}};{{U}_{0n}})={{U}_{0n}}(x)$ in ${{L}^{p}}(\mathbb{R}^N)$.
Since $\nabla \cdot ({{U}_{n}}{{b}_{n}})={{b}_{n}}\cdot \nabla {{U}_{n}}+{{U}_{n}}\nabla \cdot {{b}_{n}}$, then we have
\begin{align*}
  &\quad \nabla \cdot {{b}_{n}}(x, t)\\
&=\Delta ({{\chi }_{2}}{{v}_{2}}(x+{{x}_{n}},t+{{t}_{0n}};{{x}_{n}},{{t}_{0n}},{{u}_{0n}}(\cdot +{{x}_{n}}))-{{\chi }_{1}}{{v}_{1}}(x+{{x}_{n}},t+{{t}_{0n}};{{x}_{n}},{{t}_{0n}},{{u}_{0n}}(\cdot +{{x}_{n}}))) \\
 & =({{\chi }_{2}}{{\lambda }_{2}}{{v}_{2}}-{{\chi }_{1}}{{\lambda }_{1}}{{v}_{1}}-{{\chi }_{2}}{{\mu }_{2}}u^k+{{\chi }_{1}}{{\mu }_{1}}u^k)(x+{{x}_{n}},t+{{t}_{0n}};{{x}_{n}},{{t}_{0n}},{{u}_{0n}}(\cdot +{{x}_{n}})).
\end{align*}
Therefore, it follows that
\[
{{\partial }_{t}}{{U}_{n}}+{{(-\Delta )}^{\alpha }}{{U}_{n}}-\nabla \cdot ({{U}_{n}}{{b}_{n}})=({{A}_{n}}(x,t)-\nabla \cdot {{b}_{n}}){{U}_{n}}+\sum_{i=1,2}{{B}_{i,n}}(x,t){{V}_{n}}+\sum_{i=1,2}{{C}_{i,n}}(x,t)\cdot \nabla {{V}_{n}}
\]
for all $t>{{0}}$ and $x\in \mathbb{R}^N$. Since $\alpha \in (\frac{1}{2},1)$ and ${{e}^{t(-{{(-\Delta )}^{\alpha }}-I)}}$ is an analytic semi-group generated by $A=-{{(-\Delta )}^{\alpha }}-I$ in $X={{L}^{p}}(\mathbb{R}^N)$, then we have
\begin{align*}
  U_n(\cdot,t) &= e^{t(-(-\Delta)^\alpha - I)} U_n(0)+ \int_{t_0}^{t} e^{(t-\tau)(-(-\Delta)^\alpha - I)}
    \nabla \cdot \bigl(U_n(\cdot,\tau) b_n(\cdot,\tau)\bigr) d\tau \\
  &\quad + \int_{t_0}^{t} e^{(t-\tau)(-(-\Delta)^\alpha - I)} \bigl[
    (1 + A_n(\cdot,\tau) - \nabla \cdot b_n(\cdot,\tau)) U_n(\cdot,\tau)
  + B_{1,n}(\cdot,\tau) V_{1,n}(\cdot,\tau)\\ &\quad+ B_{2,n}(\cdot,\tau) V_{2,n}(\cdot,\tau)
  + C_{1,n}(\cdot,\tau) \cdot \nabla V_{1,n}(\cdot,\tau)+ C_{2,n}(\cdot,\tau) \cdot \nabla V_{2,n}(\cdot,\tau)
  \bigr] d\tau.
\end{align*}
For the term \( J_0 \), we have
\[
\left\| e^{t((-\Delta)^\alpha - I)} U_n(0) \right\|_{L^p(\mathbb{R}^N)} \leq e^{-t} \left\| U_n(0) \right\|_{L^p(\mathbb{R}^N)}.
\]
According to \cref{lemma:Lemma 2.5}, we obtain that
\[
\begin{aligned}
\left\|b_n(\cdot,t) \right\|_{L^\infty(\mathbb{R}^N)} &\leq \left( \frac{\chi_1 \mu_1}{\sqrt{\lambda_1}} + \frac{\chi_2 \mu_2}{\sqrt{\lambda_2}} \right) \sqrt{N} \left\| u\left( x + x_n,t + t_{0n}; x_n,t_{0n}, u_{0n}(\cdot + x_n) \right) \right\|_{L^\infty}^k \\
&\leq \left( \frac{\chi_1 \mu_1}{\sqrt{\lambda_1}} + \frac{\chi_2 \mu_2}{\sqrt{\lambda_2}} \right) \sqrt{N} C_0^k.
\end{aligned}
\]
For the second term \( J_1 \), according to \cref{lemma2.1.0}, we deduce that
\[
\begin{aligned}
\| J_1 \|_{L^p(\mathbb{R}^N)} &\leq \int_{t_0}^t e^{(t-\tau)((-\Delta)^\alpha - I)} \nabla \cdot \left( U_n(\tau, \cdot) b_n(\cdot,\tau) \right) d\tau \\
&\leq C_1 \int_{t_0}^t e^{-(t-\tau)} (t - \tau)^{-\frac{1}{2\alpha}} \| U_n(\cdot,\tau) \|_{L^p(\mathbb{R}^N)} \| b_n(\cdot,\tau) \|_{L^\infty(\mathbb{R}^N)} d\tau \\
&\leq C_1 \left( \frac{\chi_1 \mu_1}{\sqrt{\lambda_1}} + \frac{\chi_2 \mu_2}{\sqrt{\lambda_2}} \right) \sqrt{N} C_0^k \int_{t_0}^t e^{-(t-\tau)} (t - \tau)^{-\frac{1}{2\alpha}} \| U_n(\cdot,\tau) \|_{L^p(\mathbb{R}^N)} d\tau.
\end{aligned}
\]
For the remaining term \( J_2 \), since $$\sup_{\substack{0 \leq t \leq T, n \geq 1}} \| (1 + A_n(\cdot,t) - \nabla \cdot b_n(\cdot,t)) \|_{L^\infty} \leq \infty$$ and $$\sup_{\substack{0 \leq t \leq T, n \geq 1}} \| B_{i,n}(\cdot,t) \|_{L^\infty} \leq \infty,~~ \sup_{\substack{0 \leq t \leq T, n \geq 1}} \| C_{i,n}(\cdot,t) \|_{L^\infty} \leq \infty$$
for $ i = 1,2 $, which can be easily shown to be bounded, it follows that
\begin{align*}
\| J_2 \|_{L^p(\mathbb{R}^N)} &\leq \int_{t_0}^t e^{-(t-\tau)} \left[ \| 1 + A_n(\cdot,\tau) - \nabla \cdot b_n(\cdot,\tau) \|_{L^\infty} \| U_n(\tau, \cdot) \|_{L^p(\mathbb{R}^N)} \right. \\
 &\quad + \| B_{1,n}(\cdot,\tau) \|_{L^\infty} \| V_{1,n}(\cdot,\tau) \|_{L^p(\mathbb{R}^N)} + \| B_{2,n}(\cdot,\tau) \|_{L^\infty} \| V_{2,n}(\cdot,\tau) \|_{L^p(\mathbb{R}^N)} \\
&\quad + \| C_{1,n}(\cdot,\tau) \|_{L^\infty} \| V_{1,n}(\cdot,\tau) \|_{W^{1,p}(\mathbb{R}^N)} + \| C_{2,n}(\cdot,\tau) \|_{L^\infty} \| V_{2,n}(\cdot,\tau) \|_{W^{1,p}(\mathbb{R}^N)} \left. \right] d\tau \\
&\leq C_1 \int_{t_0}^t e^{-(t-\tau)} \left[ \| U_n(\cdot,\tau) \|_{L^p(\mathbb{R}^N)} + \sum_{i=1,2}\| V_{i,n}(\cdot,\tau) \|_{W^{1,p}(\mathbb{R}^N)} \right] d\tau.
\end{align*}
Given \( (\Delta - \lambda_i I) V_{i,n} = -\mu_i U_n^k \), it follows from elliptic regularity results that
\[
\| V_{i,n} \|_{W^{2,p}(\mathbb{R}^N)} \leq C \| U_n^k \|_{L^p(\mathbb{R}^N)}, \, i = 1,2.
\]
From the estimates derived above, it follows that
\[
\| J_2 \|_{L^p(\mathbb{R}^N)} \leq C_1 \int_{t_0}^t e^{-(t-\tau)} \left( \| U_n(\cdot,\tau) \|_{L^p(\mathbb{R}^N)} + \| U_n^k(\cdot,\tau) \|_{L^p(\mathbb{R}^N)} \right) d\tau.
\]
Therefore, we obtain that
\begin{equation}\label{4.16}
\begin{aligned}
&\quad \| U_n(\cdot,t) \|_{L^p(\mathbb{R}^N)}\\
&\leq e^{-t} \| U_n(0) \|_{L^p(\mathbb{R}^N)} + C_1 \left( \frac{\chi_1 \mu_1}{\sqrt{\lambda_1}} + \frac{\chi_2 \mu_2}{\sqrt{\lambda_2}} \right) C_0^k \int_{t_0}^t e^{-(t-\tau)} (t - \tau)^{-\frac{1}{2\alpha}} \| U_n(\cdot,\tau) \|_{L^p(\mathbb{R}^N)} d\tau \\
&\quad + C_1 \int_{t_0}^t e^{-(t-\tau)} \left( \| U_n(\cdot,\tau) \|_{L^p(\mathbb{R}^N)} + \| U_n^k(\cdot,\tau) \|_{L^p(\mathbb{R}^N)} \right) d\tau.
\end{aligned}
\end{equation}
In particular, we have
\begin{align*}
\| U_n^k \|_{L^p(\mathbb{R}^N)} &= \left( \int_{\mathbb{R}^N} (U_n^k)^p dx \right)^{\frac{1}{p}} \\
&= \left( \int_{\mathbb{R}^N} U_n^p (U_n^{p(k - 1)}) dx \right)^{\frac{1}{p}} \\
&\leq \sup_{x \in \mathbb{R}^N} (U_n^{p(k - 1)})^{\frac{1}{p}} \left( \int_{\mathbb{R}^N} U_n^p dx \right)^{\frac{1}{p}} \\
&= \sup_{x \in \mathbb{R}^N} (U_n^{(k - 1)}) \| U_n \|_{L^p(\mathbb{R}^N)}.
\end{align*}
By incorporating this fact, \eqref{4.16} can be rewritten
\begin{align*}
\| U_n(\cdot,t) \|_{L^p(\mathbb{R}^N)} \leq e^{-t} \| U_n(0) \|_{L^p(\mathbb{R}^N)} + C_1 \int_{t_0}^t e^{-(t-\tau)} (t - \tau)^{-\frac{1}{2\alpha}} \| U_n(\cdot,\tau) \|_{L^p(\mathbb{R}^N)} d\tau.
\end{align*}
From the conclusion of \cref{lemma:Lemma 2.99}, we can directly obtain that
\begin{align*}
\| U_n(\cdot,t) \|_{L^p(\mathbb{R}^N)} \leq C(T) \| U_n(\cdot,0) \|_{L^p(\mathbb{R}^N)}
\end{align*}
for all $0 \leq t \leq T$ and $n \geq 1$. Therefore, it follows that
\begin{equation}\label{4.17}
\begin{aligned}
\limsup_{n \to \infty} \sup_{0 \leq t \leq T} \| U_n(\cdot,t) \|_{L^p(\mathbb{R}^N)} = 0 .
\end{aligned}
\end{equation}
When \( p > N \), based on the regularity properties and a priori estimates associated with elliptic operators, we can find a positive constant \( C \) satisfying
\[
\| (\Delta - \lambda_i I) U \|_{C_{unif}^b(\mathbb{R}^N)} \leq C \| U \|_{L^p(\mathbb{R}^N)}
\]
for all $U \in L^p(\mathbb{R}^N)$. This along with \eqref{4.17} demonstrates that
\[
\limsup_{n \to \infty} \sup_{0 \leq t \leq T} \| V_{i,n}(\cdot,t) \|_{C_{unif}^{1,b}(\mathbb{R}^N)} = 0, \, i = 1,2.
\]
According to Step 3 of \cref{3.2.0}, for any $n\ge 1$, $0\le t\le T$, ${{x}_{0}}\in \mathbb{R}^N$, we have
\begin{align}\label{4.72}
  & 0\le {{\lambda }_{1}}{{v}_{1}}(x,t+{{t}_{0n}};{{t}_{0n}},{{{\tilde{u}}}_{0n}})\le \frac{{{a}_{0}}}{4{{\chi }_{1}}}, \nonumber\\
 & 0\le {{\lambda }_{2}}{{v}_{2}}(x,t+{{t}_{0n}};{{t}_{0n}},{{{\tilde{u}}}_{0n}})\le \frac{{{a}_{0}}}{4{{\chi }_{2}}}, \nonumber\\
 & |{{\chi }_{1}}\nabla {{v}_{1}}(x,t+{{t}_{0n}};{{t}_{0n}},{{{\tilde{u}}}_{0n}})|<\frac{{{\varepsilon }_{0}}}{4}, \nonumber\\
 & |{{\chi }_{2}}\nabla {{v}_{2}}(x,t+{{t}_{0n}};{{t}_{0n}},{{{\tilde{u}}}_{0n}})|<\frac{{{\varepsilon }_{0}}}{4}.
\end{align}
From \eqref{4.72}, when $n\gg 1$ we know that
\begin{align*}
  & 0\le {{\chi }_{1}}{{\lambda }_{1}}{{v}_{1}}(x,t+{{t}_{0n}};{{t}_{0n}},{{u}_{0n}}(\cdot +{{x}_{n}}))\le \frac{{{a}_{0}}}{2}, \\
 & 0\le {{\chi }_{2}}{{\lambda }_{2}}{{v}_{2}}(x,t+{{t}_{0n}};{{t}_{0n}},{{u}_{0n}}(\cdot +{{x}_{n}}))\le \frac{{{a}_{0}}}{2}, \\
 & |{{\chi }_{1}}\nabla {{v}_{1}}(x,t+{{t}_{0n}};{{t}_{0n}},{{u}_{0n}}(\cdot +{{x}_{n}}))|<\frac{{{\varepsilon }_{0}}}{2}, \\
 & |{{\chi }_{2}}\nabla {{v}_{2}}(x,t+{{t}_{0n}};{{t}_{0n}},{{u}_{0n}}(\cdot +{{x}_{n}}))|<\frac{{{\varepsilon }_{0}}}{2}.
\end{align*}
for all $ 0\le t\le T$, $x\in {{D}_{L}}$. Therefore, we have
\begin{align*}
  & |{{\chi }_{1}}{{\lambda }_{1}}{{v}_{1}}(x,t+{{t}_{0n}};{{t}_{0n}},{{u}_{0n}}(\cdot +{{x}_{n}}))-{{\chi }_{2}}{{\lambda }_{2}}{{v}_{2}}(x,t+{{t}_{0n}};{{t}_{0n}},{{u}_{0n}}(\cdot +{{x}_{n}}))|\le {{a}_{0}}, \\
 & |{{\chi }_{1}}\nabla {{v}_{1}}(x,t+{{t}_{0n}};{{t}_{0n}},{{u}_{0n}}(\cdot +{{x}_{n}}))-{{\chi }_{2}}\nabla {{v}_{2}}(x,t+{{t}_{0n}};{{t}_{0n}},{{u}_{0n}}(\cdot +{{x}_{n}}))|<{{\varepsilon }_{0}}.
\end{align*}
for all $ 0\le t\le T$, $x\in {{D}_{L}}$. Similar to the method in Step 3 of \cref{3.2.0}, we can obtain $$u(0,T+{{t}_{0n}};{{x}_{n}},{{t}_{0n}},{{u}_{0n}}(\cdot +{{x}_{n}}))>{{\delta }_{n}},$$ which contradicts the assumption. Therefore, Case 1 does not hold.\\
Case 2. $\underset{n\to \infty }{\mathop{\lim \inf }}\,|{{D}_{0n}}|>0$.

Without loss of generality, we assume $\underset{n\ge 1}{\mathop{\inf }}\,|{{D}_{0n}}|>0$ and we have $\underset{n\ge 1}{\mathop{\inf }}\,|D\cap {{D}_{0n}}|>0$ for $D\subset \subset {{D}_{3L}}$. Consider the following model
\begin{equation}\label{4.18}
\begin{cases}
u_t = -(-\Delta)^\alpha u, & x \in D_{3L}, \\
u = 0, & \text{on } (0, T) \times \partial D_{3L}, \\
u(\cdot, t_0) = \frac{\delta_0}{2} \chi_{D \cap D_{0n}}.
\end{cases}
\end{equation}
Let ${{\psi }_{n}}(x,t)$ be the solution of \eqref{4.18}. According to the comparison principle, we have
\begin{align}\label{4.19}
{{e}^{t(-{{(-\Delta )}^{\alpha }})}}{{u}_{0n}}(x+{{x}_{n}})\ge {{\psi }_{n}}(x,t)
\end{align}
for all ${{t}_{0}}\le t\le {{t}_{0}}+T$, $x\in {{D}_{3L}}$ and $n\ge 1$. From \eqref{4.19}, we obtain that
\begin{align}\label{4.20}
\left\| e^{t(-(-\Delta)^\alpha)} u_{0n}(x + x_n) \right\|_{C^\infty(D_{3L})}^2 \geq \frac{1}{|D_{3L}|} \int_{D_{3L}} \psi_n^2(x,t) \, dx
\end{align}
for all $t_0 \leq t \leq t_0 + T$, $x \in D_{3L}$ and $n \geq 1$. Note that for any $n\ge 1$, ${{\psi }_{n}}(x,t)$ can be written in the following form
\[
{{\psi }_{n}}(x,t)=\frac{{{\delta }_{0}}}{2}\sum\limits_{i=1}^{\infty }{{{e}^{-t{{\lambda }_{i}}}}{{\phi }_{i}}(x)[\int_{{{D}_{3L}}}{{{\phi }_{i}}(y){{\chi }_{D\cap {{D}_{0n}}}}(y)dy}]},
\]
where $\{\phi_i\}_{i \ge 1}$ denotes the orthonormal basis of $L^2(D_{3L})$ consisting of eigenfunctions of the fractional Laplacian $(-\Delta)^\alpha$ under Dirichlet boundary conditions, with their corresponding eigenvalues $\{\lambda_i\}$. Since $\lambda_1$ is the principal eigenvalue, its associated eigenfunction $\phi_1$ is positive throughout $D_{3L}$, i.e., $\phi_1(x) > 0$ for all $x \in D_{3L}$. Therefore, we get
\begin{align}\label{4.21}
  \left\| \psi_n(\cdot,t) \right\|_{L^2(D_{3L})}^2 &= \int_{D_{3L}} \left( \sum_{i=1}^{\infty} e^{-t\lambda_i} \phi_i(x) \left[ \frac{\delta_0}{2} \int_{D_{3L}} \phi_i(y) \chi_{D \cap D_{3L}}(y) \, dy \right] \right)^2 dx \nonumber\\
  &= \sum_{i=1}^{\infty} e^{-2t\lambda_i} \left( \int_{D_{3L}} \phi_i^2(x) \, dx \right) \left[ \frac{\delta_0}{2} \int_{D_{3L}} \phi_i(y) \chi_{D \cap D_{3L}}(y) \, dy \right]^2 \nonumber\\
  &\geq e^{-2t\lambda_1} \left[ \frac{\delta_0}{2} \int_{D_{3L}} \phi_1(y) \chi_{D \cap D_{3L}}(y) \, dy \right]^2 \nonumber\\
  &\geq e^{-2t\lambda_1} \left[ \frac{\delta_0}{2} \left| D \cap D_{0n} \right| \min_{y \in D} \phi_1(y) \right]^2.
\end{align}
Since $\underset{n\ge 1}{\mathop{\inf }}\,|D\cap {{D}_{0n}}|>0$ and $\underset{y\in D}{\mathop{\min }}\,{{\phi }_{1}}(y)>0$, combining \eqref{4.20} with \eqref{4.21}, we can obtain
\begin{equation}\label{4.22}
\begin{aligned}
\underset{{{0}}\le t\le+T}{\mathop{\inf }}\,||{{e}^{t(-{{(-\Delta )}^{\alpha }}-I)}}{{u}_{0n}}(x+{{x}_{n}})|{{|}_{C({{D}_{3L}})}}>0.
\end{aligned}
\end{equation}
According to \eqref{4.22}, there exists $0<{{T}_{0}}\ll 1$ such that
\[
\inf_{n \geq 1} \left\| u\left(\cdot,T_0 + t_{0n};x_n,t_{0n}, u_{0n}(\cdot + x_n) \right) \right\|_{C^0(D_{3L})} > 0.
\]
Therefore, we assume that \( u\left( \cdot,T_0 + t_{0n};x_n,t_{0n}, u_{0n}(\cdot + x_n) \right) \) converges locally uniformly to \( u_0^* \) and \( \left\| u_0^* \right\|_{C(D_{3L})} > 0 \).
In addition, similar to Lemma 4.2 of \cite{ZLZ1}, we can assume that
$(u(\cdot, t+{{t}_{0n}};{{x}_{n}},{{t}_{0n}},{{u}_{0n}}(\cdot +{{x}_{n}})),{{v}_{i}}(\cdot, t+{{t}_{0n}} ;{{x}_{n}},{{t}_{0n}},{{u}_{0n}}(\cdot +{{x}_{n}})))\to ({{u}^{*}}(x,t),{{v}_{i}^{*}}(x,t))$,
$a(x+{{x}_{n}},t)\to {{a}^{*}}(x,t)$, $b(x+{{x}_{n}},t)\to {{b}^{*}}(x,t)$ and $({{u}^{*}}(x,t),{{v}_{i}^{*}}(x,t))$ satisfies
\begin{equation*}
\begin{cases}
\begin{aligned}
&u_t^* = \bigl(-(-\Delta)^\alpha - I\bigr)u^*
  - \chi_1 \nabla \cdot \bigl(u^* \nabla v_1^*\bigr)
  + \chi_2 \nabla \cdot \bigl(u^* \nabla v_2^*\bigr)\\
  &\quad~~~~+ u^*\Bigl(\bigl(a^*(x,t) + 1\bigr) - b^*(x,t)u^{*,\gamma-1}\Bigr), \\[6pt]
&0 = \Delta v_1^*- \lambda_1 v_1^* + \mu_1 u^{*,k}, \\[4pt]
&0 = \Delta v_2^*- \lambda_2 v_2^* + \mu_2 u^{*,k}, \\[4pt]
&u^*(T_0) = u_0^*.
\end{aligned}
\end{cases}
\end{equation*}
Because $||u_{0}^{*}|{{|}_{\infty }}>0$ and ${{u}^{*}}(x,t)\ge 0$, according to the comparison principle of parabolic equations, we conclude that ${{u}^{*}}(x,t)>0$ for any $x\in \mathbb{R}^N$ and $t\in ({{T}_{0}},\infty )$.
In particular, ${{u}^{*}}(0,T)>0$.
However, according to \eqref{4.13}, we have ${{u}^{*}}(0,T)=0$, which contradicts the above conclusion. That is, Step 4 is proved.
\end{proof}
With the above conclusion, we prove the point persistence and uniform persistence of the solutions of \eqref{1.1}.\\
\textbf{Proof of Theorem 1.2}. Now we prove the persistence of classical solutions with strictly positive initial data for \eqref{1.1}.

\noindent Case 1. Pointwise Persistence.

Assume that condition (a) or (b) of part (i) in \cref{thm:mytheorem1.4} holds.
Fix ${{T}_{0}}>0$ and let $\delta _{0}^{*}$ and ${{M}^{+}}$ be defined as in Step 4 of \cref{3.2.0}. According to \cref{thm:mytheorem1.2}, for any ${{u}_{0}}\in C_{unif}^{b}(\mathbb{R}^N)$ with $\underset{x\in \mathbb{R}^N}{\mathop{\inf }}\,{{u}_{0}}(x)>0$, we have
\[
u(x,t;{{t}_{0}},{{u}_{0}})\le C({{u}_{0}})
\]
for all $t\ge {{t}_{0}}$, $x\in \mathbb{R}^N$.
According to \eqref{0.9.1}, there exists ${{T}_{1}}>0$ such that
\[
u(x,t;{{t}_{0}},{{u}_{0}})\le {{M}^{+}}
\]
for all $t\ge {{t}_{0}}+{{T}_{1}}$, $x\in \mathbb{R}^N$. Note that
$\underset{x\in {{R}^{N}}}{\mathop{\inf }}\,u(x,{{t}_{0}}+{{T}_{1}};{{t}_{0}},{{u}_{0}})>0$,
then there exists $0<\delta \le \delta _{0}^{*}$ such that
\begin{align*}
\delta \le u(x,{{t}_{0}}+{{T}_{1}};{{t}_{0}},{{u}_{0}})\le {{M}^{+}},
\end{align*}
for all $x\in \mathbb{R}^N$.
According to Step 4 of \cref{3.2.0}, this shows that there exists $\bar{m}({{u}_{0}})>0$ such that
\begin{align*}
\bar{m}({{u}_{0}})\le u(x,t;{{t}_{0}},{{u}_{0}})\le {{M}^{+}},
\end{align*}
for all $t\ge {{t}_{0}}$, $x\in \mathbb{R}^N$.

\noindent Case 2. Uniform Persistence.

Assume that condition (a) or (b) of part (ii) in \cref{thm:mytheorem1.4} holds.

Given an initial value data $u_0 \in C_{{unif}}^b(\mathbb{R}^N)$ with $\underset{x\in \mathbb{R}^N}{{\inf }}\,{{u}_{0}}(x)>0$, define
\[
\underline{u} = \liminf_{t \to \infty} \inf_{x \in \mathbb{R}^N} u(x, t + t_0; t_0, u_0), \quad \bar{u} = \limsup_{t \to \infty} \sup_{x \in \mathbb{R}^N} u(x, t + t_0; t_0, u_0).
\]
From Case 1, we know that $\underline{u}>0$. According to the definitions of \( \limsup \) and \( \liminf \), for any \( 0 < \varepsilon < \underline{u} \), there exists \( T_\varepsilon > 0 \) such that
\[
\underline{u} - \varepsilon \leq u(x, t; t_0, u_0) \leq \bar{u} + \varepsilon
\]
for all $x \in \mathbb{R}^N$ and $t \geq T_\varepsilon$, $i = 1, 2$. According to the comparison principle of elliptic equations, we have
\[
\mu_i (\underline{u} - \varepsilon)^k \leq \lambda_i v_i(x, t; t_0, u_0) \leq \mu_i (\bar{u} + \varepsilon)^k
\]
for all $x \in \mathbb{R}^N$ and $t \geq T_\varepsilon$, $i = 1, 2$. To prove
\[
\tilde{m} \leq \underline{u} \leq \bar{u} \leq \tilde{M},
\]
we appropriately estimate the expression for \( u \) in two different cases.\\
\textbf{Case i}. $\gamma=k+1$.
\[
\begin{aligned}
u_t &=-(-\Delta)^\alpha u - \nabla(\chi_1 v_1 - \chi_2 v_2) \cdot \nabla u + u\left(a(x + x_0,t) + \chi_2 \lambda_2 v_2 - \chi_1 \lambda_1 v_1\right) \\
&\quad- \left(b(x + x_0,t) + \chi_2 \mu_2 - \chi_1 \mu_1\right) u^{k+1} \\
&\geq-(-\Delta)^\alpha u - \nabla(\chi_1 v_1 - \chi_2 v_2) \cdot \nabla u + u\left(a_{\inf} + \chi_2 \mu_2 (\underline{u} - \varepsilon)^k - \chi_1 \mu_1 (\bar{u} + \varepsilon)^k\right) \\
&\quad- \left(b_{\sup} + \chi_2 \mu_2 - \chi_1 \mu_1\right) u^{k+1}
\end{aligned}
\]
for any \( t \geq t_0 + T_\varepsilon \). Combining with the comparison principle of parabolic equations, we have
\begin{align*}
\underline{u} &\geq \left( \frac{a_{\inf} - \chi_1 \mu_1 (\bar{u} + \varepsilon)^k + \chi_2 \mu_2 (\underline{u} - \varepsilon)^k}{b_{\sup} - \chi_1 \mu_1 + \chi_2 \mu_2} \right)^{\frac{1}{k}} \\
&\geq \left( \frac{a_{\inf} - \chi_1 \mu_1 (\bar{u} + \varepsilon)^k}{b_{\sup} - \chi_1 \mu_1 + \chi_2 \mu_2} \right)^{\frac{1}{k}}.
\end{align*}
Let \( \varepsilon \to 0 \), we get
\[
\underline{u} \geq \left( \frac{a_{\inf} - \chi_1 \mu_1 \bar{u}^k}{b_{\sup} - \chi_1 \mu_1 + \chi_2 \mu_2} \right)^{\frac{1}{k}}.
\]
According to \eqref{0.9.1}, we deduce that
\[
\bar{u}\le (\frac{{{a}_{\sup }}}{{{b}_{\inf }}-{{\chi }_{1}}{{\mu }_{1}}+{{\chi }_{2}}{{\mu }_{2}}-M})^{\frac{1}{k}}.
\]
Then we have
\[
\begin{aligned}
\underline{u} \geq \tilde{m} &= \left( \frac{a_{\inf} - \chi_1 \mu_1 \left( \frac{a_{\sup}}{b_{\inf} - \chi_1 \mu_1 + \chi_2 \mu_2 - M} \right)}{b_{\sup} - \chi_1 \mu_1 + \chi_2 \mu_2} \right)^{\frac{1}{k}} \\
&= \left( \frac{a_{\inf} \left( b_{\inf} - \left( 1 + \frac{a_{\sup}}{a_{\inf}} \right) \chi_1 \mu_1 + \chi_2 \mu_2 - M \right)}{(b_{\sup} - \chi_1 \mu_1 + \chi_2 \mu_2)(b_{\inf} - \chi_1 \mu_1 + \chi_2 \mu_2 - M)} \right)^{\frac{1}{k}}.
\end{aligned}
\]
\textbf{Case ii}. $\gamma \neq k+1$.
\[
\begin{aligned}
&\quad u_t \\
&= -(-\Delta)^\alpha u - \nabla(\chi_1 v_1 - \chi_2 v_2) \cdot \nabla u + u\left(a(x + x_0,t) + \chi_2 \lambda_2 v_2 - \chi_1 \lambda_1 v_1 - b(x + x_0,t) u^{\gamma - 1}\right) \\
&\geq -(-\Delta)^\alpha u - \nabla(\chi_1 v_1 - \chi_2 v_2) \cdot \nabla u + u\left(a_{\inf} + \chi_2 \mu_2 (\underline{u} - \varepsilon)^k - \chi_1 \mu_1 (\bar{u} + \varepsilon)^k -b_{\sup} u^{\gamma - 1}\right)
\end{aligned}
\]
for any \( t \geq t_0 + T_\varepsilon \). Combining with the comparison principle of parabolic equations, we have
\[
\begin{aligned}
\underline{u} &\geq \left( \frac{a_{\inf} - \chi_1 \mu_1 \bar{u}^k + \chi_2 \mu_2 (\underline{u} - \varepsilon)^k}{b_{\sup}} \right)^{\frac{1}{\gamma - 1}} \\
&\geq \left( \frac{a_{\inf} - \chi_1 \mu_1 \bar{u}^k}{b_{\sup}} \right)^{\frac{1}{\gamma - 1}}.
\end{aligned}
\]
According to \eqref{0.9.1}, we deduce that $\bar{u} \leq \left( \frac{a_{\sup}}{b_{\inf}} \right)^{\frac{1}{\gamma - 1}}$.
Then we have
\begin{align*}
\underline{u}\geq \tilde{m} &=( \frac{a_{\inf} - \chi_1 \mu_1 \left( \frac{a_{\sup}}{b_{\inf}} \right)^{\frac{k}{\gamma - 1}}}{b_{\sup}} )^{\frac{1}{\gamma - 1}} \\
&=( \frac{a_{\inf} ( b_{\inf}^{\frac{k}{\gamma - 1}} - \chi_1 \mu_1 \frac{a_{\sup}^{\frac{k}{\gamma - 1}}}{a_{\inf}} )}{b_{\sup} b_{\inf}^{\frac{k}{\gamma - 1}}})^{\frac{1}{\gamma - 1}}.
\end{align*}
$\hfill\Box$

\section*{Conflicts of Interest}
Authors have no conflict of interest to declare.

\section*{Acknowledgment}
This work was supported by the National Natural Science Foundation of China (11871134, 12171166).

\end{document}